\documentclass[11pt,onecolumn,draftcls]{IEEEtran}

\usepackage{cite}
\usepackage{graphicx}
\usepackage{array}
\usepackage{amsmath}
\usepackage{amssymb}
\usepackage{eurosym}
\usepackage{bm}
\usepackage{algorithm}
\usepackage{url}
\usepackage{subfigure}
\usepackage{algorithmic}

\newcommand{\set}[1]{\ensuremath{\mathcal #1}}

\newcommand{\tx}[1]{\text{tx}(l)}
\newcommand{\rx}[1]{\text{rx}(l)}

\newcommand{\separator}{
  \begin{center}
    \rule{\columnwidth}{0.3mm}
  \end{center}
}

\newcommand{\bi}{\begin{itemize}}
\newcommand{\ei}{\end{itemize}}
\newcommand{\be}{\begin{enumerate}}
\newcommand{\ee}{\end{enumerate}}

\newtheorem{theorem}{Theorem}

\newtheorem{lemma}{Lemma}
\newtheorem{definition}{Definition}

\newcommand{\beq}{\begin{eqnarray*}}
\newcommand{\eeq}{\end{eqnarray*}}
\newcommand{\beqn}{\begin{eqnarray}}
\newcommand{\eeqn}{\end{eqnarray}}

\newcommand{\argmax}{\operatornamewithlimits{argmax}} 

\title{Improved Dual Decomposition Based Optimization for DSL
Dynamic Spectrum Management}
\author{Paschalis Tsiaflakis$^*$\thanks{This research work was
carried out at the ESAT Laboratory of the Katholieke Universiteit
Leuven, in the frame of K.U. Leuven Research Council: CoE EF/05/006,
GOA AMBioRICS, FWO project G.0235.07(`Design and evaluation of DSL
systems with common mode signal exploitation'), FWO project
G.0226.06, Belgian Federal Science Policy Office IUAP DYSCO.}, Ion
Necoara, Johan A. K. Suykens, Marc Moonen\\
{\small Electrical Engineering, Katholieke Universiteit Leuven\\
Kasteelpark Arenberg 10, B-3001 Leuven-Heverlee, Belgium\\
Phone: +32 (0)16 32 18 03, Fax: +32 (0)16 32 19 70\\
email:\{paschalis.tsiaflakis, ion.necoara, johan.suykens,
marc.moonen\}@esat.kuleuven.be}\\}

\begin{document}
\maketitle

\vspace*{-18mm}
\begin{abstract}
\vspace*{-2mm}
Dynamic spectrum management (DSM) has been recognized as a key
technology to significantly improve the performance of digital
subscriber line (DSL) broadband access networks. The basic concept
of DSM is to coordinate transmission over multiple DSL lines so as
to mitigate the impact of crosstalk interference amongst
them. Many algorithms have been proposed to tackle the nonconvex
optimization problems appearing in DSM, almost all of them relying
on a standard subgradient based dual decomposition approach. In
practice however, this approach is often found to lead to
extremely slow convergence or even no convergence at all, one of
the reasons being the very difficult tuning of the stepsize
parameters. In this paper we propose a novel
improved dual decomposition approach inspired by recent
advances in mathematical programming. It uses a smoothing technique
for the Lagrangian combined with an optimal gradient based scheme
for updating the Lagrange multipliers. The stepsize parameters are
furthermore selected optimally removing the need for a tuning
strategy. With this approach we show how the convergence of current
state-of-the-art DSM algorithms based on iterative convex
approximations (SCALE, CA-DSB) can be improved by one order of
magnitude. Furthermore we apply the improved dual decomposition
approach to other DSM algorithms (OSB, ISB, ASB, (MS)-DSB, MIW) 
and propose further improvements to obtain fast and robust DSM
algorithms. Finally, we demonstrate the effectiveness of the
improved dual decomposition approach for a number of realistic
multi-user DSL scenarios. 
\end{abstract}
\vspace*{-4mm}
{\small \begin{center}
\textbf{EDICS}: SPC-TDLS Telephone networks and digital subscriber
loops, SPC-MULT Multi-carrier, OFDM, and DMT communications,
MSP-APPL Applications of MIMO communications and signal processing
\end{center}}


\section{Introduction}\label{sec:introduction}
Digital subscriber line (DSL) technology refers to a family of
technologies that provide digital broadband access over the local
telephone network. It is currently the dominating broadband access
technology with 66\% of all broadband access subscribers worldwide
using DSL to access the Internet \cite{dsldominates}. The major
obstacle for further performance improvement in modern DSL networks
is the so-called crosstalk, i.e. the electromagnetic interference
amongst different lines in the same cable bundle. Different lines
(i.e. users) indeed interfere with each other, leading to a very
challenging interference environment where proper management of the
resources is required to prevent a huge performance degradation.

Dynamic spectrum management (DSM) has been recognized as a key
technology to significantly improve the performance of DSL broadband
access networks \cite{DSM_Song}. The basic concept of DSM is to
coordinate transmission over multiple DSL lines so as to mitigate
the impact of crosstalk interference amongst them. There
are two types of coordination referred to as spectrum level and
signal level coordination. Here, we will focus on spectrum level
coordination, also referred to as spectrum management, spectrum
balancing or multi-carrier power control. Spectrum management aims
to allocate transmit spectra, i.e. transmit powers over all
available frequencies (tones), to the different users so as to
achieve some design objective. This generally corresponds to an
optimization problem, where typically a weighted sum of user data
rates is maximized subject to power constraints
\cite{optSpectrManagementJournal,dual_journal,tsiaflakis_bbosb},
which will be referred to as ``constrained weighted rate sum
maximization (cWRS)''. Recently this has been extended to other
design objectives as well, such as power driven designs (green
DSL\cite{Tsiaflakis2009a},\cite{energyDSLnordstrom},
\cite{greenCopper}) and other utility driven designs
\cite{DSMluo,Tsiaflakis2008}. As shown in \cite{Tsiaflakis2009a},
the key component to these designs is an efficient solution for the
cWRS problem. Therefore we will mainly focus on this problem and aim
to find a robust and efficient solution for it.

The cWRS problem is known to be an NP-hard, separable nonconvex
optimization problem, that can have many locally optimal solutions
\cite{DSMluo}\cite{dsb}. Even for moderately sized problems (with
5-20 users and 200-4000 tones), finding the globally optimal
solution is computationally prohibitive. In
\cite{optSpectrManagementJournal} and \cite{dual_journal} the
authors proposed to use a dual decomposition approach with a
standard subgradient based updating of the Lagrange multipliers.
Many DSM algorithms
\cite{optSpectrManagementJournal,tsiaflakis_bbosb,PBnB,ISB_Raphael,
dsb,scale,MIW,ASB,cioffi_DSB,iterativeWaterfilling} have been
proposed recently, almost all of them using this standard
subgradient based dual decomposition approach. In practice, however,
this approach is often found to lead to extremely slow convergence
or even no convergence at all, especially so for large DSL scenarios
with large crosstalk. One of the reasons is the very difficult
tuning of the stepsize parameters so as to guarantee fast
convergence.

In this paper we propose a novel improved dual decomposition
approach inspired by recent advances in mathematical
programming, more specifically the proximal center based
decomposition method recently proposed in \cite{Ion2008}. This
method uses a smoothing technique for the Lagrangian that preserves
separability of the problem, as recently proposed in
\cite{Nesterov2004}. The corresponding stepsize is determined in an
optimal way and so straightforwardly tuned. The method is
originally designed for separable convex problems, whereas DSM
optimization problems are highly nonconvex. In this paper we extend
the proximal center based decomposition method to an improved dual
decomposition approach for application in the context of DSM. With
this approach, we show how the convergence of current
state-of-the-art DSM algorithms based on iterative convex
approximations (SCALE\cite{scale}, CA-DSB\cite{dsb}) can be improved
by one order of magnitude. Furthermore we apply the improved dual
decomposition approach to other DSM algorithms
(OSB\cite{optSpectrManagementJournal},
PBnB\cite{PBnB}, ISB\cite{ISB_Raphael},
ASB\cite{ASB}, (MS-)DSB\cite{dsb}, MIW\cite{MIW},
BB-OSB\cite{tsiaflakis_bbosb}), again leading to
much faster converging DSM algorithms. Then we demonstrate an
important pitfall of applying dual decomposition to nonconvex DSM
problems and propose an effective solution for this that further
improves the robustness of current DSM algorithms. Finally we
demonstrate the effectiveness of the improved dual decomposition
approach for a number of realistic multi-user DSL scenarios.

This paper is organized as follows. In Section
\ref{sec:systemmodel}, the system model is introduced for the DSL
multi-user environment. In Section \ref{sec:dsm}, the basic cWRS
problem is described and existing DSM algorithms for this problem
are reviewed, all of them relying on a subgradient based dual
decomposition approach. In Section \ref{sec:itcvxapp} an improved
dual decomposition approach is proposed for DSM algorithms based on
iterative convex approximations. The improved dual decomposition
approach is furthermore applied to other DSM algorithms in Section
\ref{sec:generalDSM}. In Section \ref{sec:dualprimal}, the problem
of obtaining a primal solution from the dual solution is described
and an effective solution for it is proposed. Finally in Section
\ref{sec:sim_results}, simulation results are shown.\\

\section{System Model}\label{sec:systemmodel}
We consider a system consisting of $\mathcal{N}=\{1,\ldots,N\}$
interfering DSL users (i.e., lines, modems) with standard
synchronous discrete multi-tone (DMT) modulation with
$\mathcal{K}=\{1,\ldots,K\}$ tones (i.e., frequencies or carriers).
The transmission can be modeled independently on each tone $k$ by
\begin{displaymath}
\mathbf{y}_k = \mathbf{H}_k\mathbf{x}_k + \mathbf{z}_k.
\end{displaymath}
The vector $\mathbf{x}_k=[x^1_k, \ldots, x^N_k]^T$ contains
the transmitted signals on tone $k$, where $x_k^n$ refers to the
signal transmitted by user $n$ on tone $k$. Vectors $\mathbf{z}_k$
and $\mathbf{y}_k$ have similar structures; $\mathbf{z}_k$ refers
to the additive noise on tone $k$, containing thermal noise, alien
crosstalk, radio frequency interference (RFI), etc, and
$\mathbf{y}_k$ refers to the received signals on tone $k$.
$\mathbf{H}_k$ is an $N \times N$ channel matrix with
$[\mathbf{H}_k]_{n,m}=h^{n,m}_k$ referring to the channel gain from
transmitter $m$ to receiver $n$ on tone $k$. The diagonal elements
are the direct channels and the off-diagonal elements are the
crosstalk channels.

The transmit power of user $n$ on tone $k$, also referred to as
transmit power spectral density, is denoted as $s^n_k \triangleq
\Delta_f E\{|x^n_k|^2\}$, where $\Delta_f$ refers to the tone
spacing. The vector $\mathbf{s}_k \triangleq \{s_k^n, n \in
\mathcal{N}\}$ denotes the transmit powers of all users on tone
$k$. The vector $\mathbf{s}^n \triangleq \{s_k^n, k \in
\mathcal{K}\}$ denotes the transmit powers of user $n$ on all
tones. The received noise power by user $n$ on tone $k$, also
referred to as noise spectral density, is denoted as $\sigma^n_k
\triangleq \Delta_f E\{|z^n_k|^2\}$.

Note that we assume no signal coordination at the transmitters and
at the receivers, and that the interference is treated as additive
white Gaussian noise. Under this standard assumption the bit loading
for user $n$ on tone $k$, given the transmit spectra $\mathbf{s}_k$
of all users on tone $k$, is
\begin{equation} \label{eqn:bitloading}
b^n_k \triangleq b^n_k(\mathbf{s}_k) \triangleq \mathrm{log}_2
\Bigg(1 + \frac{1}{\Gamma}
\frac{|h^{n,n}_k|^2 s^n_k}{{\displaystyle \sum_{m\neq
n}}|h^{n,m}_k|^2 s^m_k + \sigma^n_k}\Bigg) ~ \mathrm{bits/Hz},
\end{equation}
where $\Gamma$ denotes the SNR-gap to capacity, which is a
function of the desired BER, the coding gain and noise margin
\cite{UnderstandingDSL}. The DMT symbol rate is denoted as $f_s$.
The achievable total data rate for user $n$ and the total power used
by user $n$ are equal to, respectively:
\begin{equation}
 \begin{array}{l}
  R^n \triangleq f_s {\displaystyle \sum_{k \in \mathcal{K}}}
b^n_k,\\
  P^n \triangleq {\displaystyle \sum_{k \in \mathcal{K}}} s^n_k.
 \end{array}
\end{equation}

 \section{Dynamic spectrum management}\label{sec:dsm}

\subsection{Dynamic spectrum management problem}
The basic goal of DSM through spectrum level coordination is to
allocate the transmit powers dynamically in response to physical
channel conditions (channel gains and noise) so as to pursue certain
design objectives and/or satisfy certain constraints. The
constraints are mostly per-user total power constraints and
so-called spectral mask constraints, i.e. 
\begin{equation}\label{eq:DSLstandard}
\begin{array}{l}
 P^n \leq P^{n,\mathrm{tot}} \qquad \qquad n \in \mathcal{N}
\qquad \qquad \ \ \
(\mathrm{total~power~constraints})\\
 0 \leq s_k^n \leq s_k^{n,\mathrm{mask}} \qquad n \in \mathcal{N},
k \in \mathcal{K} \qquad
(\mathrm{spectral~mask~constraints})
\end{array}
\end{equation}
where $P^{n,\mathrm{tot}}$ refers to the total available power
budget for user $n$ and $s_k^{n,\mathrm{mask}}$ refers to the
spectral mask constraint for user $n$ on tone $k$. The user total
power constraints can also be written in vector notation as 
$\mathbf{P} \leq \mathbf{P}^{\mathrm{tot}}$, where $\mathbf{P} =
[P^1,\ldots,P^N]^T$ and $\mathbf{P}^{\mathrm{tot}} =
[P^{1,\mathrm{tot}},\ldots,P^{N,\mathrm{tot}}]^T$, and where
'$\leq$' denotes a component-wise inequality.

The set of all possible data rate allocations that satisfy the
constraints (\ref{eq:DSLstandard}) can be characterized by the
achievable rate region $\mathcal{R}$:
\begin{equation*}\label{eq:rateregion}
\set{R} = \Big\{ (R^n: n\in \mathcal{N}) \vert R^n =  f_s \sum_{k
\in K} b^n_k(\mathbf{s}_k), \mathrm{~s.t.~(\ref{eq:DSLstandard})}
\Big\}.
\end{equation*}

A typical design objective is to achieve some Pareto optimal
allocation of the data rates $R^n$
\cite{iterativeWaterfilling,optSpectrManagementJournal, ISB_Raphael,
ISB_WeiYu,PBnB,tsiaflakis_bbosb,ASB,siwf,scale,MIW,dsb}. This
results in the following typical DSM optimization problem, which
will be referred as the constrained weighted rate sum maximization
(cWRS) formulation, where $w_n$ is the weight given to user $n$:
\begin{equation}\label{eq:DSM_wrate}
 \begin{array}{cl}
 {\displaystyle \max_{\{\mathbf{s}^n,n \in \mathcal{N}\}}} &
{\displaystyle \sum_{n \in \mathcal{N}}} w_n R^n\\
 \mathrm{s.t. } & P^n \leq P^{n,\mathrm{tot}}, \quad \qquad n \in
\mathcal{N},
\qquad \qquad \qquad (\mathrm{cWRS})\\
 & 0 \leq s^n_k \leq s^{n,\mathrm{mask}}_k, \quad n \in \mathcal{N},
k \in \mathcal{K}.
 \end{array}
 \end{equation}

However, many other DSM formulations are possible. We refer to
\cite{Tsiaflakis2009a} containing a collection of other relevant DSM
formulations. As shown in \cite{Tsiaflakis2009a}, the key component
to tackling these is an efficient solution for cWRS problem
(\ref{eq:DSM_wrate}). Therefore we will focus on this problem and
aim to find a robust and efficient solution for it.

\subsection{Dynamic spectrum management algorithms}
\label{sec:dsmalgo}
cWRS problem (\ref{eq:DSM_wrate}) is an NP-hard separable
nonconvex optimization problem \cite{DSMluo}. The number of
optimization variables is equal to $KN$, where the number of users
$N$ ranges between 2-100 and the number of tones $K$
can go up to 4000. Depending on the specific values of the channel
and noise parameters, there can be many locally optimal solutions,
that can differ significantly in value, as shown in \cite{dsb}. In
\cite{dual_journal} the authors show that strong duality holds for
the continuous (frequency range) formulation, and in \cite{DSMluo}
the authors prove asymptotic strong duality for the discrete
(frequency range) formulation, i.e. the duality gap goes to zero as
$K \rightarrow \infty$. These results suggest that a Lagrange dual
decomposition approach is a viable way to reach approximate
optimality for the discrete formulation (\ref{eq:DSM_wrate}), if the
frequency range is finely discretized, as it is indeed the case in
practical DSL scenarios where $K$ is large \cite{dual_journal}.
Many dual decomposition based DSM algorithms
\cite{optSpectrManagementJournal,tsiaflakis_bbosb,PBnB,ISB_Raphael,
dsb,scale,MIW , ASB,iterativeWaterfilling} have been proposed for
solving (\ref{eq:DSM_wrate}), almost all of them using a standard
subgradient based updating of the Lagrange multipliers.\\

The dual problem formulation of (\ref{eq:DSM_wrate}) consists of two
subproblems, namely a master problem
\begin{equation}\label{eq:master}
\begin{array}{cl}
 {\displaystyle \min_{\mathbf{\lambda}}} & g(\mathbf{\lambda})\\
 \mathrm{s.t.} & \mathbf{\lambda} \geq 0
\end{array}
\end{equation}
where $\mathbf{\lambda} = [\lambda_1, \ldots, \lambda_N]^T$, and a
slave problem defined by the dual function $g(\mathbf{\lambda})$:
\begin{equation}\label{eq:slave}
\begin{array}{l}
 g(\mathbf{\lambda}) = \mathcal{L}(\mathbf{\lambda}, \mathbf{s}_k,
k \in \mathcal{K})= \Bigg\{ \begin{array}{cl}
{\displaystyle \max_{\{\mathbf{s}_k, k \in \mathcal{K}\}}} &
{\displaystyle \sum_{n \in \mathcal{N}}} w_n R^n -
{\displaystyle \sum_{n \in \mathcal{N}}} \lambda_n (P^n -
P^{n,\mathrm{tot}})\\
\mathrm{s.t.} & 0 \leq s_k^n \leq s_k^{n,\mathrm{mask}}, \qquad
n \in \mathcal{N}, k \in \mathcal{K}, \end{array}
\end{array}
\end{equation}
where $\mathcal{L}(\mathbf{\lambda}, \mathbf{s}_k,
k \in \mathcal{K})$ is the Lagrangian. This can be
reformulated as:
\begin{equation}\label{eq:slave2}
\begin{array}{l}
 g(\mathbf{\lambda}) = \Bigg\{ \begin{array}{cl}
{\displaystyle \max_{\{\mathbf{s}_k, k \in \mathcal{K}\}}} &
{\displaystyle \sum_{k \in \mathcal{K}}} \Big\{ f_s
b_k(\mathbf{s_k}) - {\displaystyle \sum_{n \in \mathcal{N}}}
\lambda_n s_k^n \Big\} + {\displaystyle \sum_{n \in \mathcal{N}}}
\lambda_n P^{n,\mathrm{tot}}\\
\mathrm{s.t.} & 0 \leq s_k^n \leq s_k^{n,\mathrm{mask}}, \qquad
n \in \mathcal{N}, k \in \mathcal{K} \end{array}\\
\mathrm{with~~} b_k(\mathbf{s}_k) = {\displaystyle \sum_{n \in
\mathcal{N}}} w_n b_k^n(\mathbf{s}_k)
\end{array}
\end{equation}

The slave optimization problem (\ref{eq:slave2}) can then be
decomposed into $K$ independent nonconvex subproblems (dual
decomposition):
\begin{equation}\label{eq:slaves}
\begin{array}{l}
g(\mathbf{\lambda})={\displaystyle \sum_{k \in K}}
g_k(\mathbf{\lambda})\\
\mathrm{with~}g_k(\mathbf{\lambda}) = \Bigg\{
\begin{array}{cl}
{\displaystyle \max_{\mathbf{s}_k}} & f_s b_k(\mathbf{s}_k) -
{\displaystyle \sum_{n \in \mathcal{N}}} \lambda_n s_k^n +
{\displaystyle \sum_{n \in \mathcal{N}}} \lambda_n
P^{n,\mathrm{tot}}/K\\
\mathrm{s.t.} & 0 \leq s_k^n \leq s_k^{n,\mathrm{mask}}, \quad
n \in \mathcal{N}
\end{array}\\
\end{array}\\
\end{equation}

The master problem (\ref{eq:master}), also called the dual problem,
is a convex optimization problem. Its objective function, i.e. the 
dual function $g(\mathbf{\lambda})$, is however non-differentiable.
The reason for this non-differentiability is that the underlying
slave optimization problem (\ref{eq:slave}) can have multiple
globally optimal solutions for some values of the Lagrange
multipliers $\mathbf{\lambda}$. In
\cite{dual_journal}\cite{optSpectrManagementJournal} a subgradient
approach is proposed for this dual master problem, where the
subgradient is defined as,
\begin{equation}\label{eq:subgrad}
dg(\mathbf{\lambda}) \triangleq \sum_{k \in \mathcal{K}}
\mathbf{s}_k(\mathbf{\lambda}) - \mathbf{P}^{\mathrm{tot}}
\end{equation}
with $\mathbf{s}_k(\mathbf{\lambda})$ referring to the optimal
solution of (\ref{eq:slaves}) for given
Lagrange multipliers $\mathbf{\lambda}$, also called dual variables,
and the corresponding subgradient update is:
\begin{equation}\label{eq:subgradient}
 \mathbf{\lambda} = \Big[\mathbf{\lambda} + \delta (\sum_{k
\in \mathcal{K}} \mathbf{s}_k(\mathbf{\lambda}) -
\mathbf{P}^{\mathrm{tot}})\Big]^+,
\end{equation}
where $[\mathbf{x}]^+$ denotes the projection of $\mathbf{x} \in
\mathcal{R}^N$ onto $\mathcal{R}^N_+$, and where the stepsize
$\delta$ can be chosen using different procedures
\cite{dual_journal,tsiaflakis_bbosb}, e.g. $\delta = q/i$ where $q$
is the initial stepsize and $i$ is the iteration counter. By
iteratively applying (\ref{eq:subgradient}) and (\ref{eq:slaves}),
convergence to an optimal solution of (\ref{eq:master}) can be
achieved, i.e. $\mathbf{\lambda} \rightarrow \mathbf{\lambda}^*$,
for which the complementary conditions, $\lambda_n (\sum_{k \in
\mathcal{K}} \mathbf{s}^n_k(\mathbf{\lambda}) -
\mathbf{P}^{n,\mathrm{tot}}) = 0,\ n \in \mathcal{N}$, are satisfied
when strong duality ``holds'' ($K \rightarrow \infty$). This general
standard dual decomposition approach is visualized in Figure
\ref{fig:generalDSM}.

Note that the per-tone subproblems (\ref{eq:slaves}) are nonconvex
optimization problems. Many existing DSM algorithms differ only in
the way these subproblems are solved, where strategies are proposed
such as exhaustive discrete search
(OSB)\cite{optSpectrManagementJournal}, branch and bound search
(PBnB\cite{PBnB}, BB-OSB\cite{tsiaflakis_bbosb}), coordinate
descent discrete search (ISB)\cite{ISB_Raphael}\cite{ISB_WeiYu},
solving the KKT system (DSB\cite{dsb}, MIW\cite{MIW},
MS-DSB\cite{dsb}), and heuristic approximation (ASB\cite{ASB},
ASB2\cite{dsb}). 

\begin{figure}[!h]
\centering
\scalebox{0.6}{\input{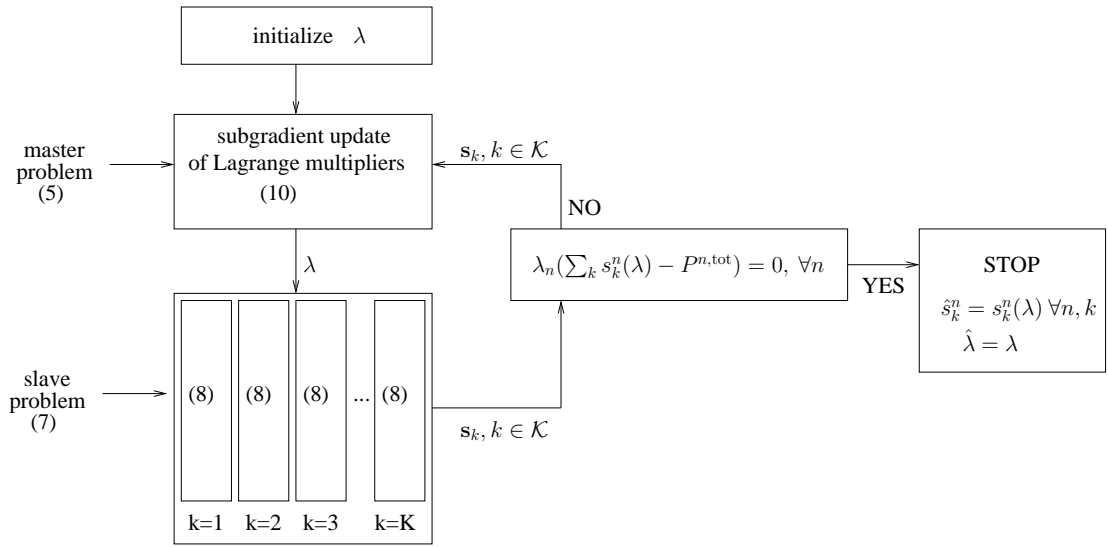}}
\caption{General structure of subgradient based dual decomposition
approach for DSM}
\label{fig:generalDSM}
\end{figure}

An alternative approach is based on iterative convex
approximations such as in SCALE\cite{scale} and CA-DSB\cite{dsb}.
This approach basically consists of iteratively executing the
following two steps: (i) approximating the nonconvex cWRS problem
(\ref{eq:DSM_wrate}) by a separable convex optimization problem
$F_{\mathrm{cvx}}$, and (ii) solving this convex
approximation by using a subgradient based dual decomposition
approach. Note that under some conditions on the approximation,
described in \cite{Chiang2007}, iteratively executing these steps
results in asymptotic convergence to a locally optimal solution of
cWRS (\ref{eq:DSM_wrate}). The convex approximations used by CA-DSB
and SCALE both satisfy these conditions. This approach is visualized
in Figure \ref{fig:itcvxDSM}, where $F_{k,\mathrm{cvx}}$ refers to
the per-tone convex problem obtained from the convex approximation
$F_{\mathrm{cvx}}$. We emphasize that these DSM algorithms
also use a subgradient based dual decomposition approach to solve a
convex optimization problem in each iteration. This step requires
the major part of the computational cost.

\begin{figure}[!h]
\centering
\scalebox{0.6}{\input{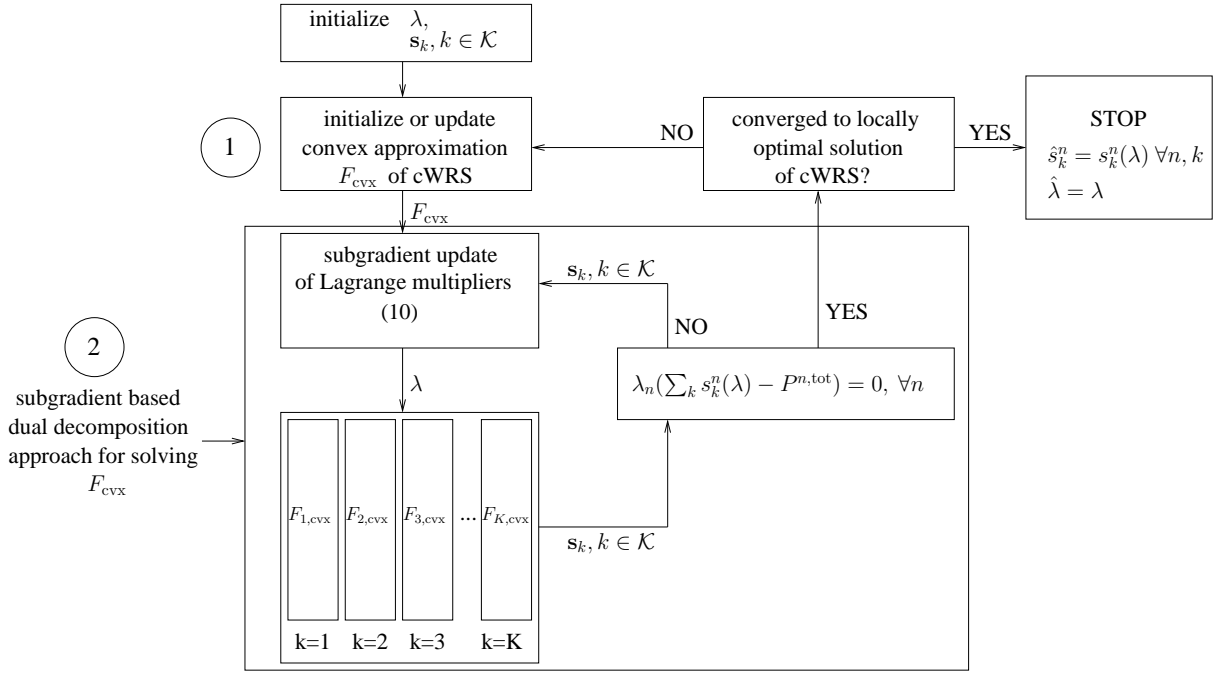}}
\caption{Structure of iterative convex approximation approach for
DSM}
\label{fig:itcvxDSM}
\end{figure}

\section{Improved dual decomposition}\label{sec:idd}
In practice, the standard subgradient based dual decomposition
approach is often found to lead to extremely slow convergence or
even no convergence at all, especially so for large DSL scenarios
(6-20 users) with large crosstalk (VDSL(2)). This is because of
different reasons: (i) subgradient methods are generally known not
to be efficient, i.e. showing worst case convergence of order
$\mathcal{O}(\frac{1}{\epsilon^2})$ with $\epsilon$ referring to the
required accuracy of the approximation of the optimum
\cite{Nesterov2004}, (ii) the stepsize used by subgradient methods
is quite difficult to tune in order to guarantee fast convergence,
(iii) the nonconvex nature of the problem implies that special care
should be taken in obtaining the optimal primal variables from the
optimal dual variables.

For separable convex problems, i.e. with a separable convex
objective function but with convex coupling constraints, several
alternative dual decomposition approaches have been proposed such as
the alternating direction method \cite{KonLeo:96}, proximal method
of multipliers \cite{CheTeb:94}, partial inverse method
\cite{Spi:85}, etc. Here, we focus on a recently proposed dual
decomposition approach in \cite{Ion2008}, referred to as the
proximal center based decomposition method. This method shows
interesting properties, namely it preserves separability of the
problem, it uses an optimal gradient based scheme, and it uses an
optimal stepsize which leads to straightforward tuning. In this
section we extend this method to
an improved dual decomposition approach for solving cWRS
(\ref{eq:DSM_wrate}). This approach will be used first in Section
\ref{sec:itcvxapp} to improve the convergence of DSM algorithms
using iterative convex approximations (SCALE, CA-DSB) with one order
of magnitude. In Section \ref{sec:generalDSM} this will be extended
to other DSM algorithms such as OSB, ISB, PBnB, BB-OSB, ASB,
(MS-)DSB, MIW, etc. We will refer to these DSM algorithms that are
not based on iterative convex approximations as ``direct DSM
algorithms''.


\subsection{An improved dual decomposition approach for iterative
convex approximation based DSM algorithms}\label{sec:itcvxapp} Two
state-of-the-art DSM algorithms that are based on iterative convex
approximations are SCALE and CA-DSB. These basically consist of two
steps as explained in Section \ref{sec:dsmalgo}, which are
iteratively executed. In this section we will propose an improved
dual decomposition approach for solving the convex optimization
problem in the second step. We will elaborate this for CA-DSB and
explain how its convergence speed is improved by one order of
magnitude, i.e. from $\mathcal{O}(\frac{1}{\epsilon^2})$ to 
$\mathcal{O}(\frac{1}{\epsilon})$.
The improved dual decomposition approach can similarly be applied to
the SCALE algorithm to obtain a similar speed up, but requires more
complicated notation because of the inherent exponential
transformation of variables.\\

For CA-DSB, the convex approximation in each iteration is obtained
by reformulating the objective of cWRS, as a sum of a concave part
and a convex part, and then approximating the convex part
by a first order Taylor expansion. The resulting convex
approximation, its dual formulation, dual function, and Lagrangian
are given in (\ref{eq:cvxapproxprimal}),
(\ref{eq:cvxapprox}), (\ref{eq:dualfunction}), and
(\ref{eq:lagrangian}), respectively.
\begin{eqnarray}
& f^*_{\mathrm{cvx}} = \big\{ {\displaystyle \max_{\{\mathbf{s}_k
\in \mathcal{S}_k, k \in \mathcal{K}\}}} {\displaystyle \sum_{k \in
\mathcal{K}}} b_{k,\mathrm{cvx}}(\mathbf{s}_k) \ \ \mathrm{s.t. }
{\displaystyle \sum_{k \in \mathcal{K}}} s_k^n \leq
P^{n,\mathrm{tot}}, \ n \in \mathcal{N}
\big\}\label{eq:cvxapproxprimal}\\
& {\displaystyle \min_{\mathbf{\lambda} \geq 0}}
\ g_{\mathrm{cvx}}(\mathbf{\lambda})\label{eq:cvxapprox}\\
& g_{\mathrm{cvx}}(\mathbf{\lambda}) =
{\displaystyle \max_{\{\mathbf{s}_k \in \mathcal{S}_k, k \in
\mathcal{K}\}}} \mathcal{L}_{\mathrm{cvx}}(\mathbf{s}_k, k \in
\mathcal{K}, \mathbf{\lambda})\label{eq:dualfunction}\\
& \mathcal{L}_{\mathrm{cvx}}(\mathbf{s}_k, k \in \mathcal{K},
\mathbf{\lambda}) = {\displaystyle \sum_{k \in \mathcal{K}}}
b_{k,\mathrm{cvx}}(\mathbf{s}_k) - {\displaystyle \sum_{k \in
\mathcal{K}}} \ {\displaystyle \sum_{n \in \mathcal{N}}} \lambda_n
s_k^n +  {\displaystyle \sum_{n \in \mathcal{N}}} \lambda_n
P^{n,\mathrm{tot}}
\label{eq:lagrangian}
\end{eqnarray}
where $\mathcal{S}_k = \{\mathbf{s}_k \in \mathcal{R}^n: 0 \leq
s_k^n \leq s_k^{n,\mathrm{max}}, n \in \mathcal{N} \}$ is a compact
convex set with
$s_k^{n,\mathrm{max}}:=\min(s_k^{n,\mathrm{mask}},P^{n,\mathrm{tot}}
)$ and $P^{n,\mathrm{tot}}<\infty$, and where $
b_{k,\mathrm{cvx}}(\mathbf{s}_k)$ is concave and given as:
\begin{equation}
 \begin{array}{l}
b_{k,\mathrm{cvx}}(\mathbf{s}_k) = {\displaystyle \sum_{n \in
\mathcal{N}}} w_n f_s \log_2({\displaystyle \sum_{m \in
\mathcal{N}}} |\tilde{h}_k^{n,m}|^2 s_k^m + \Gamma \sigma_k^n) -
{\displaystyle \sum_{n \in \mathcal{N}}} w_n f_s
({\displaystyle \sum_{m \neq n}} a_k^{m,n}s_k^m+c_k^n),\\
 \end{array}
\end{equation}
with $a_k^{n,m}, c_k^n, \forall n,m,k$ constant
approximation parameters, obtained by a closed-form formula in the
approximation step \cite{dsb}, and with
\begin{equation}
\begin{array}{l}
 |\tilde{h}_k^{n,m}|^2 \Bigg\{ \begin{array}{l} = \Gamma
|h_k^{n,m}|^2, \qquad n \neq m\\
 = |h_k^{n,m}|^2, \qquad n = m.
\end{array}
\end{array}
\end{equation}

The convex problem (\ref{eq:cvxapproxprimal}) has a separable
structure and so the standard way to solve it is by focusing on the
dual problem (\ref{eq:cvxapprox}) and using a subgradient update
approach for the dual variables. This subgradient based dual
decomposition approach is however known \cite{Nesterov2004} to have
a convergence speed of order $\mathcal{O}(\frac{1}{\epsilon^2})$,
where $\epsilon$ is the required accuracy for the approximation of
the optimum. In the sequel, it will be shown how the ``proximal
center based decomposition'' method from \cite{Ion2008} can be
adapted for solving the convex approximation, leading to a
scheme with convergence speed of order
$\mathcal{O}(\frac{1}{\epsilon})$, i.e. one order of magnitude
faster  but with the same computational complexity. The basic steps
in this result are as follows. First an approximated (smoothed) dual
function $\bar{g}_{\mathrm{cvx}}(\mathbf{\lambda})$ is defined that
can be chosen to be arbitrarily close to the original dual function
$g_{\mathrm{cvx}}(\mathbf{\lambda})$. Then it is proven that this
smoothed dual function $\bar{g}_{\mathrm{cvx}}$ is differentiable
and has a Lipschitz continuous gradient. Finally an optimal gradient
scheme is applied to this smoothed dual function.\\

We introduce the following functions $d_k(\mathbf{s}_k), k \in
\mathcal{K}$, which are called prox-functions in \cite{Ion2008} and
are defined as
follows:

\begin{definition}
A prox-function $d_k(\mathbf{s}_k)$ has the following properties:
\begin{itemize}
 \item $d_k(\mathbf{s}_k)$ is a non-negative continuous and strongly
convex function with convexity parameter $\sigma_{\mathcal{S}_k}$
 \item $d_k(\mathbf{s}_k)$ is defined for the compact convex set
$\mathcal{S}_k$
\end{itemize}
\end{definition}

An example of a valid prox-function is $d_k(\mathbf{s}_k)
=\frac{1}{2} \Vert\mathbf{s}_k\Vert^2$, which is also used in our
concrete implementations (see Section \ref{sec:sim_results}). As
many other valid prox-functions exist, and in order not to loose
generality, we continue with $d_k(\mathbf{s}_k)$. Since
$\mathcal{S}_k, k \in \mathcal{K},$ are compact and
$d_k(\mathbf{s}_k)$ are continuous, we can choose
finite and positive constants such that
\begin{equation}
 D_{\mathcal{S}_k} \geq \max_{\mathbf{s}_k \in \mathcal{S}_k}
d_{k}(\mathbf{s}_k), k \in \mathcal{K}.
\end{equation}

The prox-functions can be used to smoothen the dual function
$g_{\mathrm{cvx}}(\mathbf{\lambda})$ to obtain a smoothed dual
function $\bar{g}_{\mathrm{cvx}}(\mathbf{\lambda})$ as follows:
\begin{equation}\label{eq:approx_dual}
 \bar{g}_{\mathrm{cvx}}(\mathbf{\lambda}) = \max_{\{\mathbf{s}_k \in
\mathcal{S}_k, k \in \mathcal{K}\}} \sum_{k \in \mathcal{K}} \Big\{
b_{k,\mathrm{cvx}}(\mathbf{s}_k) - \sum_{n \in \mathcal{N}}
\lambda_n (s_k^n - P^{n,\mathrm{tot}}/K)- c d_{k}(\mathbf{s}_k)
\Big\},
\end{equation}
where $c$ is a positive smoothness parameter that will be defined
later in this section. By using a sufficiently  small value for $c$,
the smoothed dual function can be arbitrarily close to original dual
function. Note that we can also choose different parameters $c_k$
for each prox-term. The generalization is straightforward.

One useful property of the particular choice of prox-functions is
that they do not destroy the separability of the
objective function in (\ref{eq:approx_dual}), i.e.
\begin{equation}\label{eq:approx_dual2}
 \bar{g}_{\mathrm{cvx}}(\mathbf{\lambda}) = \sum_{k \in \mathcal{K}}
\bigg\{ \max_{\mathbf{s}_k \in \mathcal{S}_k}
b_{k,\mathrm{cvx}}(\mathbf{s}_k)
- \sum_{n \in \mathcal{N}} \lambda_n (s_k^n - P^{n,\mathrm{tot}}/K)-
c d_{k}(\mathbf{s}_k) \bigg\}.
\end{equation}

Denote by $\bar{\mathbf{s}}_{k,\mathrm{cvx}}(\mathbf{\lambda}),
k \in \mathcal{K},$ the optimal solution of the maximization problem
in (\ref{eq:approx_dual2}). The following theorem describes the
properties of the smoothed dual function
$\bar{g}_{\mathrm{cvx}}(\mathbf{\lambda})$:
\begin{theorem}[\hspace{-0.01mm}\cite{Ion2008}]\label{theorem:ineq}
The function $\bar{g}_{\mathrm{cvx}}(\mathbf{\lambda})$ is convex
and continuously differentiable at any $\mathbf{\lambda} \in
\mathcal{R}^n$. Moreover, its gradient $\nabla
\bar{g}_{\mathrm{cvx}}(\mathbf{\lambda}) = \sum_{k \in \mathcal{K}}
\bar{\mathbf{s}}_{k,\mathrm{cvx}}(\mathbf{\lambda}) -
\mathbf{P}^{\mathrm{tot}}$ is Lipschitz continuous with Lipschitz
constant $L_c = \sum_{k \in \mathcal{K}} \frac{1}{c
\sigma_{\mathcal{S}_k}}$.
The following inequalities also hold:
\begin{equation}
 \bar{g}_{\mathrm{cvx}}(\mathbf{\lambda}) \leq
g_{\mathrm{cvx}}(\mathbf{\lambda}) \leq
\bar{g}_{\mathrm{cvx}}(\mathbf{\lambda}) + c \sum_{k \in
\mathcal{K}} D_{\mathcal{S}_k} \qquad \mathbf{\lambda} \in
\mathcal{R}^n
\end{equation}
\end{theorem}

The addition of the prox-functions thus leads to a convex
differentiable dual function with Lipschitz continuous gradient. Now
instead of solving the original dual problem (\ref{eq:cvxapprox}),
we focus on the following problem
\begin{equation}\label{eq:smoothdualproblem}
 \min_{\mathbf{\lambda} \geq 0}
\bar{g}_{\mathrm{cvx}}(\mathbf{\lambda})
\end{equation}
Note that, by making $c$ sufficiently small in
(\ref{eq:approx_dual2}), the solution of
(\ref{eq:smoothdualproblem}) can be made arbitrarily close to the
solution of (\ref{eq:cvxapprox}). Taking the particular structure
of (\ref{eq:smoothdualproblem}) into account, i.e. a differentiable
objective function with Lipschitz continuous gradient, we propose
the optimal gradient based scheme given in Algorithm
\ref{algo:icadsb}, derived from \cite{Ion2008}, for solving
(\ref{eq:cvxapproxprimal}). This algorithm will
be referred to as the \emph{improved dual decomposition algorithm}
for solving the convex approximation of CA-DSB
(\ref{eq:cvxapproxprimal}).

\begin{algorithm}
\caption{Improved dual decomposition algorithm for solving
(\ref{eq:cvxapproxprimal}) for CA-DSB}\label{algo:icadsb}
\begin{algorithmic}[1]
\STATE $i := 0$, tmp $:=0$ \STATE initialize $i_{\mathrm{max}}$
\STATE initialize $\mathbf{\lambda}^i$ \FOR{$i=0 \ldots
i_{\mathrm{max}}$}
\STATE $\forall k: \mathbf{s}_k^{i+1} =
{\displaystyle \argmax_{\{\mathbf{s}_k \in
\mathcal{S}_k\}}} \
b_{k,\mathrm{cvx}}(\mathbf{s}_k) - {\displaystyle \sum_{n \in
\mathcal{N}}} \mathbf{\lambda}^i_n s_k^n - c d_{k}(\mathbf{s}_k)$
\STATE $d\bar{g}_c^{i+1} = {\displaystyle \sum_{k \in \mathcal{K}}}
\mathbf{s}^{i+1}_k - \mathbf{P}^{\mathrm{tot}}$ \STATE
$\mathbf{u}^{i+1} =
[\frac{d\bar{g}_c^{i+1}}{L_c}+\mathbf{\lambda}^i]^+$ \STATE
$\mathrm{tmp} := \mathrm{tmp} + \frac{i+1}{2} d\bar{g}_c^{i+1}$
\STATE
$\mathbf{v}^{i+1} = [\frac{\mathrm{tmp}}{L_c}]^+$ \STATE
$\mathbf{\lambda}^{i+1} = \frac{i+1}{i+3} \mathbf{u}^{i+1} +
\frac{2}{i+3} \mathbf{v}^{i+1}$ \STATE $i:= i+1$ \ENDFOR \STATE
Build $\hat{\mathbf{\lambda}} =
\mathbf{\lambda}^{i_{\mathrm{max}}+1}$ and $\hat{\mathbf{s}}_k =
\sum_{i=0}^{i_{\mathrm{max}}}
\frac{2(i+1)}{(i_{\mathrm{max}}+1)(i_{\mathrm{max}}+2)}\mathbf{s}^{
i+1}_k$
\end{algorithmic}
\end{algorithm}

The specific value for $L_c$ depends on the chosen prox-function
$d_k(\mathbf{s}_k)$, as given in Theorem \ref{theorem:ineq}. The
specific value for $c$ will be defined later in Theorem
\ref{theorem:gap}. Note that lines $6$-$10$ of Algorithm
\ref{algo:icadsb} correspond to the improved Lagrange multiplier
updates. By comparing this with the standard subgradient Lagrange
multiplier update (\ref{eq:subgradient}), one can observe that the
standard and improved update require a similar complexity.

The remaining issue is to prove that $\hat{\mathbf{s}}_k, k \in
\mathcal{K},$ converges to an $\epsilon$-optimal solution in
$i_{\mathrm{max}}$ iterations where $i_{\mathrm{max}}$ is of the
order $\mathcal{O}(\frac{1}{\epsilon})$. For this we define the
following lemmas that will be used in the sequel.

\begin{lemma}
\label{genCS} For any $\mathbf{y} \in \mathcal{R}^n$ and $\mathbf{z}
\geq 0$, the following inequality holds\footnote{For the sake of an
easy exposition we consider in the paper only the Euclidian norm
$\Vert \cdot \Vert$, although other norms can also  be used (see
\cite{Ion2008} for a detailed exposition).}:
\begin{equation}
 \mathbf{y}^T \mathbf{z} \ \leq  \Vert \mathbf{[y]}^+\Vert
\Vert\mathbf{z}\Vert
\end{equation}
\end{lemma}
\begin{proof}
Let us define the index sets $\mathcal{I}^{-} = \{i \in
\{1\ldots n\}: y_i<0 \}$ and $\mathcal{I}^{+} = \{i \in \{1\ldots
n\}: y_i \geq 0 \}$. Then,
\begin{align*}
\mathbf{y}^T \mathbf{z} = \sum_{i \in \mathcal{I}^{-}} y_i z_i +
\sum_{i \in \mathcal{I}^{+}} y_i z_i \leq  \sum_{i \in
\mathcal{I}^{+}} y_i z_i = ([\mathbf{y}]^+)^T \mathbf{z} \leq \Vert
\mathbf{[y]}^+\Vert \Vert\mathbf{z}\Vert.
\end{align*}
\\
\end{proof}

The following lemma provides a lower bound for the primal gap,
$f^*_{\mathrm{cvx}} - \sum_{k \in \mathcal{K}}
b_{k,\mathrm{cvx}}(\hat{\mathbf{s}}_k)$, of
(\ref{eq:cvxapproxprimal}):

\begin{lemma}\label{lemma:primal} Let $\mathbf{\lambda}^*$ be any
optimal Lagrange multiplier, then for any $\hat{\mathbf{s}}_k \in
\mathcal{S}_k, k \in \mathcal{K}$, the following lower bound on the
primal gap holds:
\begin{equation}\label{eq:primalgap}
f^*_{\mathrm{cvx}} - \sum_{k \in \mathcal{K}}
b_{k,\mathrm{cvx}}(\hat{\mathbf{s}}_k) \geq -
\Vert\mathbf{\lambda^*}\Vert  \Vert[\sum_{k \in \mathcal{K}}
\hat{\mathbf{s}}_k - \mathbf{P^{\mathrm{tot}}}]^+\Vert
\end{equation}
\end{lemma}
\begin{proof}
From the assumptions of the lemma we have
 \begin{equation} f^*_{\mathrm{cvx}} = \max_{\{{\mathbf{s}}_k \in
\mathcal{S}_k, k \in \mathcal{K}\}} \sum_{k \in
\mathcal{K}} b_{k,\mathrm{cvx}}({\mathbf{s}}_k)  -
\mathbf{\lambda^*}^T (\sum_{k \in \mathcal{K}} {\mathbf{s}}_k -
\mathbf{P^{\mathrm{tot}}}) \geq \sum_{k \in \mathcal{K}}
b_{k,\mathrm{cvx}}(\hat{\mathbf{s}}_k)  -
\mathbf{\lambda^*}^T (\sum_{k \in \mathcal{K}} \hat{\mathbf{s}}_k -
\mathbf{P^{\mathrm{tot}}})
\end{equation}
and then (\ref{eq:primalgap}) is obtained by applying Lemma
\ref{genCS}.\\
\end{proof}

From Lemma \ref{lemma:primal} it follows that if $\Vert[\sum_{k
\in \mathcal{K}} \hat{\mathbf{s}}_k -
\mathbf{P^{\mathrm{tot}}}]^+\Vert \leq \epsilon_c$, then the primal
gap is bounded, i.e. for all $\hat{\mathbf{\lambda}} \in
\mathcal{R}^N_+$
\begin{equation}
 - \epsilon_c \Vert\mathbf{\lambda}^*\Vert \leq f^*_{\mathrm{cvx}} -
\sum_{k \in \mathcal{K}} b_{k,\mathrm{cvx}}(\hat{\mathbf{s}}_k) \leq
g_{\mathrm{cvx}}(\hat{\mathbf{\lambda}}) - \sum_{k \in \mathcal{K}}
b_{k,\mathrm{cvx}}(\hat{\mathbf{s}}_k).
\end{equation}

Therefore, if we are able to derive an upper bound $\epsilon$ for
the dual gap, namely $g_{\mathrm{cvx}}(\hat{\mathbf{\lambda}}) -
\sum_{k \in \mathcal{K}} b_{k,\mathrm{cvx}}(\hat{\mathbf{s}}_k)$,
and an upper bound $\epsilon_c$ for the coupling constraints for
some given $\hat{\mathbf{\lambda}}$ ($\geq 0$) and
$\hat{\mathbf{s}}_k \in \mathcal{S}_k, \forall k,$ then we
can conclude that $\hat{\mathbf{s}}_k$ is an ($\epsilon,
\epsilon_c$)-solution for (\ref{eq:cvxapproxprimal}) (since in this
case $-\epsilon_c \Vert\mathbf{\lambda}^*\Vert \leq
f^*_{\mathrm{cvx}} - \sum_{k \in \mathcal{K}}
b_{k,\mathrm{cvx}}(\hat{\mathbf{s}}_k) \leq
\epsilon$). The next theorem derives these upper bounds for
Algorithm \ref{algo:icadsb} and provides a concrete value for $c$.

\begin{theorem}\label{theorem:gap}
 Let $\mathbf{\lambda}^*$ be an optimal Lagrange multiplier, taking
$c=\frac{\epsilon}{\sum_{k \in \mathcal{K}} D_{\mathcal{S}_k}}$ and
\newline
$i_{\mathrm{max}}+1 = 2 \sqrt{(\sum_k
\frac{1}{\sigma_{\mathcal{S}_k}})(\sum_k
D_{\mathcal{S}_k})}\frac{1}{\epsilon}$, then after
$i_{\mathrm{max}}$ iterations Algorithm \ref{algo:icadsb} obtains
an approximate solution $\hat{\mathbf{s}}_k, k \in \mathcal{K},$ to
the convex approximation (\ref{eq:cvxapproxprimal}) with a duality
gap less than $\epsilon$, i.e.
\begin{equation}
\label{gap} g_{\mathrm{cvx}}(\hat{\mathbf{\lambda}}) - \sum_{k \in
\mathcal{K}} b_{k,\mathrm{cvx}}(\hat{\mathbf{s}}_k) \leq \epsilon,
\end{equation}
and the constraints satisfy
\begin{equation}
\label{ineq-gap} \Vert [\sum_k \hat{\mathbf{s}}_k -
\mathbf{P}^{\mathrm{tot}}]^+\Vert \leq \epsilon
(\Vert\mathbf{\lambda}^*\Vert+\sqrt{\Vert\mathbf{\lambda}^{*}
\Vert^2+2} )
\end{equation}
\end{theorem}
\begin{proof}
Using a similar reasoning as in Theorem 3.4 in \cite{Ion2008} we
can show that for any $c$ the following inequality holds:
\[
\bar{g}_{\mathrm{cvx}}(\hat{\mathbf{\lambda}}) \leq
\min_{{\mathbf{\lambda}} \geq 0}  \big \{ \frac{2
L_c}{(i_{\mathrm{max}}+1)^2}  \Vert {\mathbf{\lambda}} \Vert^2 +
\sum_{i=0}^{i_{\mathrm{max}}}
\frac{2(i+1)}{(i_{\mathrm{max}}+1)(i_{\mathrm{max}}+2)}
[\bar{g}_{\mathrm{cvx}}(\lambda^i) + (\nabla
\bar{g}_{\mathrm{cvx}}(\lambda^i))^T (\lambda - \lambda^i)] \big\}
\]
Replacing  $\bar{g}_{\mathrm{cvx}}(\lambda^i)$ and $\nabla
\bar{g}_{\mathrm{cvx}}(\lambda^i)$ by their expressions given in
(\ref{eq:approx_dual}) and Theorem \ref{theorem:ineq},
respectively, and taking into account that the functions
$b_{k,\mathrm{cvx}}$ are concave, we obtain the following
inequality:
\begin{align*}
g_{\mathrm{cvx}}(\hat{\mathbf{\lambda}}) - \sum_{k \in \mathcal{K}}
b_{k,\mathrm{cvx}}(\hat{\mathbf{s}}_k) & \leq  c(\sum_{k \in
\mathcal{K}} D_{\mathcal{S}_k}) + \min_{{\mathbf{\lambda}} \geq 0} 
\big \{
\frac{2 L_c}{(i_{\mathrm{max}}+1)^2}  \Vert {\mathbf{\lambda}}
\Vert^2 - \langle  \lambda , \sum_k \hat{\mathbf{s}}_k -
\mathbf{P}^{\mathrm{tot}} \rangle      \big \}\\
 & = c(\sum_{k \in
\mathcal{K}} D_{\mathcal{S}_k}) - \frac{(i_{\mathrm{max}}+1)^2}{8
L_c} \Vert
[\sum_k \hat{\mathbf{s}}_k - \mathbf{P}^{\mathrm{tot}}]^+  \Vert^2
\leq c(\sum_{k \in \mathcal{K}} D_{\mathcal{S}_k}).
\end{align*}
By taking $c=\frac{\epsilon}{\sum_{k \in \mathcal{K}}
D_{\mathcal{S}_k}}$, we obtain \eqref{gap}. For
the constraints using Lemma \ref{lemma:primal}  and the previous
inequality  we get that $  \Vert [\sum_k \hat{\mathbf{s}}_k -
\mathbf{P}^{\mathrm{tot}}]^+\Vert$ satisfies the second order
inequality in $y$:  $ \frac{(i_{\mathrm{max}}+1)^2}{8 L_c} y^2  -
\|\lambda^*\| y -\epsilon \leq 0$. Therefore, $\Vert [\sum_k
\hat{\mathbf{s}}_k - \mathbf{P}^{\mathrm{tot}}]^+\Vert$ must be less
than the largest root of the corresponding second-order equation,
i.e.
\[ \Vert [\sum_k \hat{\mathbf{s}}_k -
\mathbf{P}^{\mathrm{tot}}]^+\Vert \leq \big ( \| \lambda^*\| +
\sqrt{\|\lambda^*\|^2 + \frac{\epsilon(i_{\mathrm{max}}+1)^2}{2
L_c}} \big ) \frac{4 L_c}{(i_{\mathrm{max}}+1)^2}.
\]
By taking $i_{\mathrm{max}}=2 \sqrt{(\sum_k
\frac{1}{\sigma_{\mathcal{S}_k}})(\sum_k
D_{\mathcal{S}_k})}\frac{1}{\epsilon}-1$, we obtain
\eqref{ineq-gap}.
\end{proof}
~\\
From Theorem \ref{theorem:gap} we can conclude that by taking
$c=\frac{\epsilon}{\sum_{k \in \mathcal{K}} D_{\mathcal{S}_k}}$,
Algorithm
\ref{algo:icadsb} converges to a solution with duality gap less than
$\epsilon$ and the constraints violation satisfy $\Vert [\sum_k
\hat{\mathbf{s}}_k - \mathbf{P}^{\mathrm{tot}}]^+\Vert \leq \epsilon
(\Vert\mathbf{\lambda}^*\Vert+\sqrt{\Vert\mathbf{\lambda}^{*}
\Vert^2+2})$ after $i_{\mathrm{max}} = 2 \sqrt{(\sum_k
\frac{1}{\sigma_{\mathcal{S}_k}})(\sum_k
D_{\mathcal{S}_k})}\frac{1}{\epsilon}-1$
iterations, i.e. the convergence speed is of the order
$\mathcal{O}(\frac{1}{\epsilon})$.

Note that Algorithm \ref{algo:icadsb} provides a fully automatic
approach, i.e. it requires no stepsize tuning, which is otherwise
known to be a very difficult and crucial process. Finally note that
combining this algorithm with an outer loop that iteratively updates
the convex approximations leads to an overall procedure that
converges to a local maximizer of the nonconvex problem cWRS
\cite{Chiang2007}\cite{dsb}. The extension of CA-DSB with the
improved dual decomposition approach will be referred to as
Improved CA-DSB (I-CA-DSB).\\

A final remark on Algorithm \ref{algo:icadsb} is that the
independent convex per-tone problems (line 5 of Algorithm
\ref{algo:icadsb}) are slightly modified with respect to the
standard per-tone problems for CA-DSB. This is a consequence of the
addition of the extra prox-function term. One can use
state-of-the-art iterative methods (e.g. Newton's method) to solve
this convex subproblem with guaranteed convergence. An alternative
consists in using an iterative fixed point update approach, which is
shown to work well, with very small complexity, and is easily
extended to distributed implementation by using a protocol
\cite{scale}\cite{dsb}. The fixed point update formula for the
transmit powers $s_k^n$ used by CA-DSB can be adapted so
as to take the extra prox-term into account. Following the same
procedure as explained in \cite{dsb}, consisting of a fixed point
reformulation of the corresponding KKT stationarity condition of
(\ref{eq:cvxapprox}), we obtain the following transmit power update
formula, that only differs in the presence of the term PROX:
\begin{equation}\label{eq:cadsb_power}
s_k^n  = \Bigg[ \bigg( \frac{w_n
f_s/\log(2)}{\lambda_n + \underbrace{2 c s_k^n}_{\mathrm{PROX}}
+{\displaystyle \sum_{m \neq n}} \omega_m f_s a_k^{n,m}-
{\displaystyle \sum_{m\neq n}} \frac{w_m f_s \Gamma
|h_k^{m,n}|^2/\log(2)}{{\displaystyle \sum_p} |\tilde{h}_k^{m,p}|^2
s_k^p+ \Gamma \sigma_k^m}} \bigg) - \frac{{\displaystyle \sum_{m
\neq n}} \Gamma |h_k^{n,m}|^2 s_k^{m} + \Gamma
\sigma_k^n}{|h_k^{n,n}|^2}\Bigg]_0^{s_k^{n,\mathrm{mask}}}.
\end{equation}

Providing convergence conditions for this type of iterative fixed
point updates is outside the scope of this paper. In
\cite{ASB,dsb,iterativeWaterfilling}, convergence is proven under
certain conditions, and demonstrated for realistic DSL scenarios.
This leads to an alternative and fast way of implementing line $5$
of Algorithm \ref{algo:icadsb}, as specified in Algorithm
\ref{algo:itfixpoint}. The number of iterations in line 2 is
typically fixed at $3$. Note that a distributed solution is also
possible for the full scheme as the dual decomposition approach is
decoupled over the users. (see \cite{dsb} for more details). 
\begin{algorithm}
\caption{Iterative fixed point update approach for solving line
$5$ of Algorithm \ref{algo:icadsb}}\label{algo:itfixpoint}
\begin{algorithmic}[1]
\FOR{$k = 1 \ldots K$}
\FOR{iterations}
\FOR{$n = 1 \ldots N$}
\STATE $s_k^n = $(\ref{eq:cadsb_power})
\ENDFOR
\ENDFOR
\ENDFOR
\end{algorithmic}
\end{algorithm}

As mentioned, although the improved dual decomposition approach
has been elaborated for CA-DSB, it can similarly be applied to
other DSM algorithms based on iterative convex approximations, like
for instance SCALE, with a similar speed up of convergence. In this
case the prox-function can be taken as $d_k(\mathbf{s}_k) = \Vert
\mathbf{s}_k \Vert^2$, resulting in concrete values for $c$,
$i_{\mathrm{max}}$ and $L_c$. The extension of SCALE with the
improved dual decomposition approach will be referred to as Improved
SCALE (I-SCALE).

\subsection{An improved dual decomposition approach for direct DSM
algorithms}\label{sec:generalDSM}

In this section we extend the improved dual decomposition
approach to direct DSM algorithms such as OSB, ISB, ASB, (MS-)DSB,
MIW, etc, corresponding to the structure visualized in Figure
\ref{fig:generalDSM}. Using a similar trick as in Section
\ref{sec:itcvxapp}, we define a smoothed dual function
$\bar{g}(\mathbf{\lambda})$ as follows
\begin{equation}
\bar{g}(\mathbf{\lambda}) = \Bigg\{
\begin{array}{cl}
{\displaystyle \max_{\mathbf{s}_k \in \mathcal{S}_k, k \in
\mathcal{K}}} & {\displaystyle \sum_{k \in \mathcal{K}}} f_s
b_k(\mathbf{s}_k) - {\displaystyle \sum_{k \in \mathcal{K}} \sum_{n
\in \mathcal{N}}} \lambda_n s_k^n + {\displaystyle \sum_{n \in
\mathcal{N}}} \lambda_n P^{n,\mathrm{tot}} - {\displaystyle \sum_{k
\in \mathcal{K}}} c d_k(\mathbf{s}_k)\\
\mathrm{s.t.} & 0 \leq s_k^n \leq s_k^{n,\mathrm{mask}}, \quad
k \in \mathcal{K}, n \in \mathcal{N},
\end{array}
\end{equation}
where $d_k(\mathbf{s}_k)$ is a prox-function, which for
instance can be chosen as $d_k(\mathbf{s}_k)=\Vert \mathbf{s}_k
\Vert^2$, and $c=\frac{\epsilon}{\sum_{k \in \mathcal{K}}
D_{\mathcal{S}_k}}$,
with $\epsilon$ the required accuracy, and $L_c = \sum_{k \in
\mathcal{K}} \frac{1}{c \sigma_{\mathcal{S}_k}}$.

Note that by choosing a sufficiently small value for $c$, the
smoothed dual function $\bar{g}(\mathbf{\lambda})$ can be made
arbitrarily close to the original dual function
$g(\mathbf{\lambda})$,
i.e. $\bar{g}(\mathbf{\lambda}) \approx g(\mathbf{\lambda})$.

This results in the \emph{improved dual decomposition approach for
direct DSM algorithms}, given in Algorithm \ref{algo:idsm}, where
line 4 uses the following optimization problem:
\begin{equation}\label{eq:islaves}
\begin{array}{ll}
\tilde{\mathbf{s}}_k(\mathbf{\lambda}) =
& {\displaystyle \argmax_{\mathbf{s}_k}} f_s
b_k(\mathbf{s}_k) - {\displaystyle \sum_{n \in \mathcal{N}}}
\lambda_n s_k^n - c d_k(\mathbf{s}_k)\\
& \mathrm{s.t.~} 0 \leq s_k^n \leq s_k^{n,\mathrm{mask}}, \quad
n \in \mathcal{N},
\end{array}
\end{equation}

\begin{algorithm}
\caption{Improved dual decomposition approach for direct
DSM algorithms}\label{algo:idsm}
\begin{algorithmic}[1]
\STATE $i := 0$, tmp $:=0$
\STATE initialize $\mathbf{\lambda}^i$ and $\epsilon_a$ (desired
accuracy)
\WHILE{$\exists n:(\mathrm{abs}(\lambda_n^i
({\displaystyle \sum_{k \in \mathcal{K}}} s_k^n-P^{n,\mathrm{tot}}))
\geq \epsilon_a)$}
\STATE $\forall k:$ $\mathbf{s}_k^{i+1} =
\tilde{\mathbf{s}}_k(\mathbf{\lambda^i})$ obtained by solving
(\ref{eq:islaves})
\STATE $dg^{i+1} = {\displaystyle \sum_{k \in \mathcal{K}}}
\mathbf{s}^{i+1}_k - \mathbf{P}^{\mathrm{tot}}$
\STATE $\mathbf{u}^{i+1} =
[\frac{dg^{i+1}}{L_c}+\mathbf{\lambda}^i]^+$
\STATE $\mathrm{tmp} := \mathrm{tmp} + \frac{i+1}{2} dg^{i+1}$
\STATE $\mathbf{v}^{i+1} = [\frac{\mathrm{tmp}}{L_c}]^+$
\STATE $\mathbf{\lambda}^{i+1} = \frac{i+1}{i+3} \mathbf{u}^{i+1}
+ \frac{2}{i+3} \mathbf{v}^{i+1}$
\STATE $i := i + 1$
\ENDWHILE
\STATE Build $\hat{\mathbf{\lambda}} = \mathbf{\lambda}^{i}$ and
$\hat{\mathbf{s}}_k = \mathbf{s}_k^{i}, \forall k \in
\mathcal{K}$
\end{algorithmic}
\end{algorithm}

Algorithm \ref{algo:idsm} uses a similar optimal gradient based
scheme on the smoothed Lagrangian as in Algorithm \ref{algo:icadsb}.
Again no stepsize tuning is needed. Besides the improved
updating procedure for the Lagrange multipliers (lines 5-9), it
involves a slightly different decomposed per-tone problem
(\ref{eq:islaves}) (line 4). This can be solved by using a discrete
exhaustive search similar to OSB, a discrete coordinate descent
method similar to ISB, or a KKT system approach similar to
DSB/MIW/MS-DSB using (\ref{eq:cadsb_power}), where
$a_k^{n,m} = \frac{\Gamma |h_k^{m,n}|^2/\log(2)}{\sum_{p \neq m}
\Gamma |h_k^{m,p}|^2 s_k^p + \Gamma \sigma_k^m}$ \cite{dsb}. One can
also use a virtual reference length approach similar to ASB, ASB2.
Note that for ASB, and when using $d_k(\mathbf{s}_k)=\Vert
\mathbf{s}_k\Vert^2$, this increases the complexity as a polynomial
equation of degree 4 is then to be solved instead of a cubic
equation. Depending on the choice of the algorithm for solving the
per-tone problem, there will be a trade-off in complexity versus
performance \cite{dsb}. We will again add the prefix 'I-' to refer
to these algorithms using the improved dual decomposition approach,
i.e. I-OSB,I-ISB, I-DSB/MIW, I-MS-DSB, I-ASB.\\

The main difference of Algorithm \ref{algo:idsm} is that line 4
now involves $K$ nonconvex optimization problems, while line 5 of
Algorithm \ref{algo:icadsb} involves $K$ (strong) convex
optimization problems. As a consequence, the smoothed dual function
$\bar{g}(\mathbf{\lambda})$ is not necessarily differentiable and
its gradient is not necessarily Lipschitz continuous. More
specifically, this is the case when $\bar{g}(\mathbf{\lambda})$ has
multiple globally optimal solutions for a given Lagrange multiplier
$\mathbf{\lambda}$. This specific condition however mainly occurs
for a particular type of DSL scenarios which are analyzed and
discussed in Section \ref{sec:dualprimal}. For these scenarios the
worst case convergence of order $\mathcal{O}(\frac{1}{\epsilon})$
can not be guaranteed, as in Theorem \ref{theorem:gap}, but still we
can expect an improved convergence behaviour with respect to the
standard subgradient approach. Except for these specific cases, and
so for most practical DSL scenarios, the smoothed dual
function $\bar{g}(\mathbf{\lambda})$ will be differentiable and
Lipschitz continuous, and so a worst case convergence speed of
$\mathcal{O}(\frac{1}{\epsilon})$ is guaranteed. For instance, in
\cite{Tsiaflakis2008a} conditions on the channel and noise
parameters were given under which cWRS can be ``convexified''. For
these conditions, differentiability and Lipschitz continuity holds
for $\bar{g}(\mathbf{\lambda})$ and so application of Algorithm
\ref{algo:idsm} will provide a worst case convergence of
$\mathcal{O}(\frac{1}{\epsilon})$.


\section{An interleaving procedure for recovering the primal
solution from the dual solution}\label{sec:dualprimal}

The subgradient based dual decomposition approach for solving
problem cWRS (\ref{eq:DSM_wrate}) as well as the improved dual
decomposition approach presented in Sections \ref{sec:itcvxapp} and
\ref{sec:generalDSM}, converge to the optimal dual variables.
However, because of the nonconvex nature of cWRS, extra care must be
taken when recovering the optimal primal solution, i.e. optimal
transmit powers $\mathbf{s}_k^*, k \in \mathcal{K},$ for
(\ref{eq:DSM_wrate}), from the optimal dual variables
$\mathbf{\lambda}^*$, as was also mentioned in
\cite{dual_journal}\cite{Boyd2004}. The fact that the objective
function of cWRS is not strictly concave, can result in cases where
the optimal $\mathbf{s}_k(\mathbf{\lambda}^*), k \in \mathcal{K},$
that solves (\ref{eq:slave2}) is not unique, leading to multiple
solutions $\mathbf{s}_k(\mathbf{\lambda}^*), k \in \mathcal{K},$ for
given optimal dual variables $\mathbf{\lambda}^*$. Formally this can
be expressed as follows:
\begin{equation}
 \begin{array}{l}
 \{\mathbf{s}_k(\mathbf{\lambda}^*), k \in \mathcal{K}\} \in
\mathcal{B}=\{(\tilde{\mathbf{s}}_{k,1}, k \in
\mathcal{K}), \ldots ,(\tilde{\mathbf{s}}_{k,|\mathcal{B}|}, k \in
\mathcal{K})\}\\ \mathrm{~with~} \tilde{\mathbf{s}}_{k,m} \in
\mathcal{S}_k, k \in \mathcal{K}, \mathrm{~and~}
\mathcal{L}(\tilde{\mathbf{s}}_{k,m}, k \in \mathcal{K},
\mathbf{\lambda}^*) = {\displaystyle \max_{\{\mathbf{s}_k \in
\mathcal{S}_k, k \in \mathcal{K}\}}} \mathcal{L}(\mathbf{s}_k, k \in
\mathcal{K}, \mathbf{\lambda}^*), \quad m \in \{1, \ldots,
|\mathcal{B}| \},
 \end{array}
\end{equation}
where the cardinality of set $\mathcal{B}$ is larger than 1, i.e.
$|\mathcal{B}|>1$. It is important to note that the elements of
$\mathcal{B}$ are not necessarily solutions to (\ref{eq:DSM_wrate}),
i.e. they do not necessarily satisfy the user total power
constraints (\ref{eq:DSLstandard}). However, there exists at least
one element in set $\mathcal{B}$ that does satisfy the total power
constraints \cite{dual_journal}. In order to obtain convergence to a
primal optimal solution for (\ref{eq:DSM_wrate}) in the case that
$|\mathcal{B}|>1$, the dual decomposition approach has to be
extended with an extra procedure that chooses an element out of set
$\mathcal{B}$ that satisfies the user total power constraints. 

A simple example may be given to clarify this issue; suppose we have
a DSL scenario consisting of two users ($N=2$) and two tones
($K=2$), where the channel matrices (direct and crosstalk
components) and noise components for the two tones are the same,
i.e. $\mathbf{H}_1 = \mathbf{H}_2$ and $\sigma_1^n=\sigma_2^n, \ n
\in \mathcal{N}$, and the weights are also the same $w_1 = w_2$.
Furthermore suppose the crosstalk components are very large. In this
case, there will be only one user active on each tone
\cite{luo_dsm}. Finally suppose that the optimal dual variables
$\lambda_1^*,\lambda_2^*$, where $\lambda_1^* = \lambda_2^*$, are
given and the total power constraints are $P^n\leq \mathrm{ON}$,
where $\mathrm{ON}$ is a fixed power level. For this setup there
will be 4 possible solutions to (\ref{eq:slave2}), namely
$\{s_1^1=\mathrm{ON},s_2^1=\mathrm{ON},s_1^2=0,s_2^2=0\},\{s_1^1=0,
s_2^1=0,s_1^2=\mathrm{ON},s_2^2=\mathrm{ON}\},
\{s_1^1=\mathrm{ON},s_2^1=0,s_1^2=0,s_2^2=\mathrm{ON}\},
\{s_1^1=0,s_2^1=\mathrm{ON},s_1^2=\mathrm{ON},s_2^2=0\}$. Note that
all these solutions correspond to exactly the same objective value
but only the last two solutions are primal optimal solutions 
as they satisfy the user total power constraints. Typical DSM
algorithm implementations, however, have a fixed exhaustive search
order or iteration order over tones so that one of the two first
solutions may be selected and, as a consequence, these algorithms
will not provide the primal optimal solutions of
(\ref{eq:DSM_wrate}). To obtain convergence to the optimal primal
variables of (\ref{eq:DSM_wrate}) an extra procedure should be added
to the dual decomposition approach.

Note that the above problem is practically only relevant when the
phenomenon of non-unique globally optimal solutions
$\mathbf{s}_k(\mathbf{\lambda}^*)$ occurs at many tones. This is the
case for DSL scenarios that have a subset of strong symmetric
crosstalkers with equal line lengths, i.e. lines that generate the
same interference to their environment over multiple tones $k$, with
equal weights $w_n$ and user total power constraints
$P^{n,\mathrm{tot}}$. Here, we can have many subsequent tones with
multiple globally optimal solutions, namely where only one of the
subset of strong crosstalkers is active \cite{luo_dsm}. If no
special care is taken when recovering the primal transmit powers,
this can lead to extremely slow convergence or even no convergence
at all for these scenarios. More specifically, a fixed exhaustive
search order or iteration order in typical DSM algorithm
implementations will choose the same strong crosstalker over all
competing tones, instead of equally dividing the resources over the
competing users. 

To overcome this problem we propose a very simple, but effective,
interleaving procedure. More specifically this solution consists of
alternatingly on a per-tone basis, giving priority to the globally
optimal solution that corresponds to a different active strong
crosstalker of the symmetric subset. This interleaving procedure
replaces line 4 of Algorithm \ref{algo:idsm} with the following:
\begin{equation}\label{eq:interleaving}
 \forall k : \left\{
\begin{array}{ll}
 \mathcal{C}_k &= \{\mathrm{all~globally~optimal~solutions~
\tilde{\mathbf{s}}_k(\mathbf
{\lambda})~of~(\ref{eq:islaves}) \mathrm{~for~given~}
\mathbf{\lambda} }\},\\
& = \{\mathcal{C}_k(1),\ldots,\mathcal{C}_k(|\mathcal{C}_k|)\},\\
\mathrm{index} &= \mathrm{rem}(k,|\mathcal{C}_k|)+1, \\
\mathbf{s}_k^{i+1} &= \mathcal{C}_k (\mathrm{index}),
\end{array}
\right.
\end{equation}
where `rem($k,|\mathcal{C}_k|$)' refers to the remainder after
dividing $k$ by $|\mathcal{C}_k|$. As the suggested solution
requires that all globally optimal solutions in the first step of
(\ref{eq:interleaving}) actually be computed, it should be
combined with algorithms for the per-tone nonconvex problem that
indeed compute all these solutions such as OSB with a fixed order
exhaustive search for all tones or a multiple starting point
approach such as MS-DSB with a fixed iteration order for all tones. 

In the simulation Section \ref{sec:sim_results}, it will be
demonstrated how the usage of (\ref{eq:interleaving}) significantly
improves the robustness of the dual decomposition approach for
cWRS.\\

\textbf{Remark:} The above mentioned non-uniqueness also has an
impact on the Lipschitz continuity condition of the smoothed
gradient. More specifically this condition reduces to
\cite{Ion2008}:
\begin{equation}\label{eq:LC}
\Vert \sum_{k \in \mathcal{K}}
\tilde{\mathbf{s}}_k(\mathbf{\lambda}) - \sum_{k \in \mathcal{K}}
\tilde{\mathbf{s}}_k(\mathbf{\mu}) \Vert^2 \leq L_c \Vert
\mathbf{\lambda} - \mathbf{\mu} \Vert^2 \mathrm{~with~} L_c <
\infty
\end{equation}

For the above two-user two-tone symmetric strong crosstalk example,
this condition does not hold. This can be shown as follows. Let us
compare two cases: (1) optimal dual variables
$(\lambda_1^*,\lambda_2^*+\mu)$ corresponding to primal variables
$\{s_1^1=\mathrm{ON},s_2^1=\mathrm{ON},s_1^2=0,s_2^2=0\}$,
(2) optimal dual variables $(\lambda_1^*+\mu,\lambda_2^*)$
corresponding to primal variables
$\{s_1^1=0,s_2^1=0,s_1^2=\mathrm{ON},s_2^2=\mathrm{ON}\}$, where
$\mathbf{\mu} \geq 0$. For very small $\mu$ these two cases have
only slightly different dual variables but completely
different primal variables. So a small change in Lagrange
multipliers can lead to a large change in primal variables. This
means that for these specific cases Lipschitz continuity
(\ref{eq:LC}) is not satisfied and so the convergence speed will
be worse than $O(\frac{1}{\epsilon})$. However adding the
interleaving trick alleviates this problem, as will be demonstrated
in Section \ref{sec:sim_results}.

%
%

\section{Simulation results}\label{sec:sim_results}

In this section, simulation results are shown that compare the
performance of the improved dual decomposition approach with respect
to the subgradient based dual decomposition approach. More
specifically, in Section \ref{sec:simItCvxApp} we demonstrate the
convergence speed-up in using the improved dual decomposition
approach with respect to the subgradient based dual decomposition
approach for a DSM algorithm based on iterative convex
approximations (CA-DSB). In Section \ref{sec:simDirDSM} we
demonstrate how the improved dual decomposition approach in
combination with a direct DSM algorithm (MS-DSB) succeeds in
providing much faster convergence than with the subgradient based
dual decomposition approach. Furthermore the convergence improvement
for the interleaving procedure presented in Section
\ref{sec:dualprimal} is demonstrated.

The following parameter settings are used for the simulated DSL
scenarios. The twisted pair lines have a diameter of 0.5 mm (24
AWG). The maximum per-user total transmit power is 11.5 dBm for the
VDSL scenarios and 20.4 dBm for the ADSL scenarios. The SNR gap
$\Gamma$ is 12.9 dB, corresponding to a coding gain of 3 dB, a noise
margin of 6 dB, and a target symbol error probability of $10^{-7}$.
The tone spacing $\Delta_f$ is 4.3125 kHz. The DMT symbol rate $f_s$
is 4 kHz.

\subsection{Convergence speed up for iterative convex approximation
based DSM}\label{sec:simItCvxApp}
A first DSL scenario is shown in Figure \ref{fig:adsl}. This is a
so-called near-far scenario which is known to be challenging, where
DSM can make a substantial difference. For this scenario, we compare
the convergence behaviour for the improved approach for CA-DSB
(Algorithm \ref{algo:icadsb}) and the standard subgradient based
dual decomposition approach for CA-DSB, where convergence is defined
as achieving the optimal dual value of the convex approximation
within accuracy 0.05\%. The results are shown in Figure
\ref{fig:convergence}. For the subgradient scheme we used the
stepsize update rule $\delta=q/i$, where $q$ is the initial
stepsize and $i$ is the iteration counter \cite{dual_journal}. This
update rule is proven to converge to the optimal dual value. It can
be observed that different initial stepsizes lead to a different
convergence behaviour and this is generally difficult to tune. Note
that for all initial stepsizes, the subgradient dual decomposition
approach is still far from convergence after 500 iterations. The
improved dual decomposition approach, on the other hand,
automatically tunes its stepsize and converges very rapidly in only
40 iterations.

\begin{figure}[h]
\centering
\includegraphics[width=0.4\columnwidth]{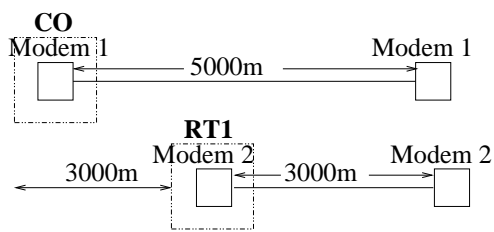}
\caption{2-user near-far ADSL downstream scenario}
\label{fig:adsl}
\end{figure}

\begin{figure}[h]
\centering
\includegraphics[width=0.6\columnwidth]{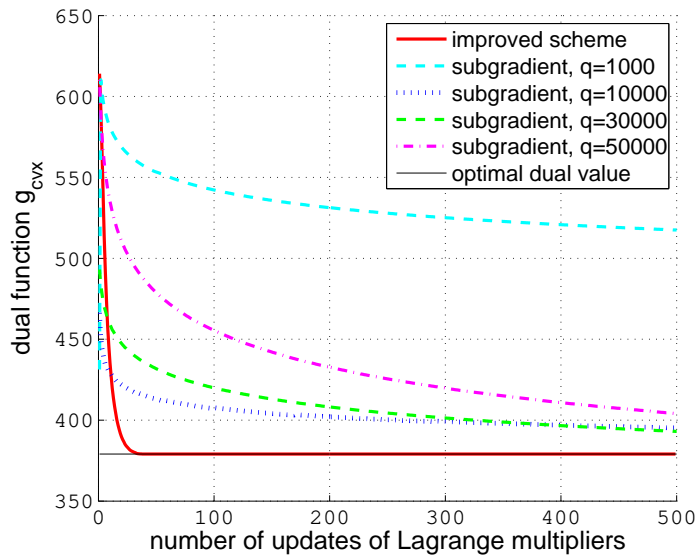}
\caption{Comparison of convergence behaviour between subgradient
dual decomposition approach, with different initial stepsizes $q$,
and the improved dual decomposition approach, for CA-DSB}
\label{fig:convergence}
\end{figure}

\subsection{Convergence speed up for direct
DSM}\label{sec:simDirDSM}

It was shown in \cite{tsiaflakis_bbosb} that for direct DSM
algorithms the subgradient based dual decomposition approach with a
particular stepsize selection procedure works well for ADSL
scenarios, i.e. there are typically only 50-100 subgradient
iterations needed to converge to the optimal dual variables. However
for multi-user VDSL scenarios, which use a much larger frequency
range and have to cope with significantly more crosstalk
interference, existing subgradient approaches
\cite{dual_journal}\cite{tsiaflakis_bbosb} are found to have
significant convergence problems. We will focus on such VDSL
scenarios and demonstrate how the improved approach succeeds in
providing much faster convergence.\\

The different VDSL scenarios are shown in Figures \ref{fig:vdsl4},
\ref{fig:vdsl6} and \ref{fig:vdsl6sym}, i.e. four-user VDSL
upstream, six-user VDSL upstream, and six-user VDSL upstream
scenario with a subset of strong symmetric crosstalkers,
respectively. The weights $w_n$ are chosen equal for all users $n$,
namely $w_n=1/N$. Note that we used the multiple starting point
procedure MS-DSB to solve the nonconvex per-tone problems for the
subgradient based dual decomposition approach as well as the
improved dual decomposition approach using (\ref{eq:cadsb_power}).
In \cite{dsb} it was shown that this procedure provides globally
optimal performance for practical ADSL and VDSL scenarios.\\

The first scenario, shown in Figure \ref{fig:vdsl4}, is a four-user
upstream VDSL scenario, consisting of two far-users with line length
$1200$ m and two near-users with line length $300$ m. In the higher
frequency range, there is a significant crosstalk coupling. This is
a near-far scenario where spectrum management is crucial as to avoid
significant performance degradation for the far-end users. Note that
the near-end users form a subset of strong symmetric crosstalkers,
in the high frequency range. As mentioned in Section
\ref{sec:dualprimal}, this can cause significant convergence
problems for the dual decomposition approach. In fact, simulations
show that the subgradient methods in \cite{tsiaflakis_bbosb} and
\cite{dual_journal} fail to converge to the dual variables, i.e.
after $20000$ iterations the complementarity conditions for some
users are far from being satisfied. The main problem is that the
stepsize selection procedure, which is a crucial component for fast
convergence, is difficult to tune. For decreasing step sizes as
proposed in \cite{dual_journal}, with different initial stepsizes,
the procedure does not converge. For adaptive stepsizes, as proposed
in \cite{tsiaflakis_bbosb}, very small stepsizes are selected
resulting in a very slow convergence ($> 20000 $ iterations). It is
observed that for some users there is a fast convergence to the
corresponding complementarity conditions whereas for other users
convergence is very slow. The presence of the strong subset of
symmetric crosstalkers, can lead to large changes in primal
variables for small changes in dual variables, as discussed in
Section \ref{sec:dualprimal}, if stepsizes are not tuned carefully.
The improved approach of Algorithm \ref{algo:idsm}, in contrary,
converges very fast to the optimal dual and primal variables. In
only $100$ iterations convergence is obtained, within an accuracy of
$0.05\%$.\\ 

\begin{figure}[h]
\centering
\includegraphics[width=0.5\columnwidth]{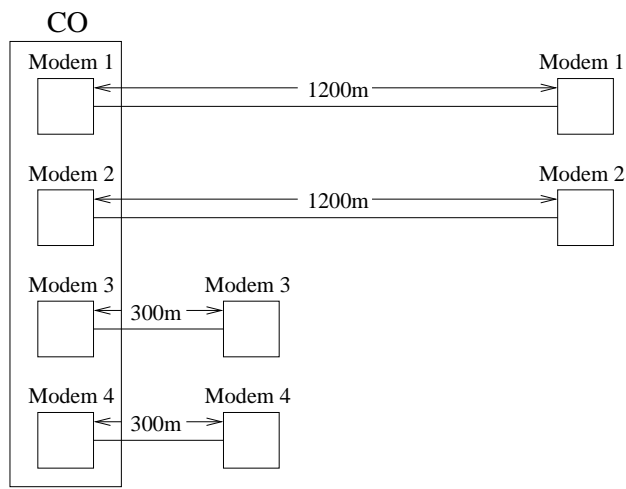}
\caption{4-user VDSL upstream scenario}
\label{fig:vdsl4}
\end{figure}

The second VDSL upstream scenario, shown in Figure \ref{fig:vdsl6},
consists of six users with different line lengths.
Also for this large crosstalk scenario, the standard subgradient
approaches \cite{tsiaflakis_bbosb}\cite{dual_journal} fail to
converge to the optimal dual variables, i.e. after $10000$
iterations the complementarity conditions are far from being
satisfied. Similarly to the scenario of Figure \ref{fig:vdsl4}, one
can observe very different convergence behaviour for the different
users to the corresponding complementarity conditions, where
typically for a few users convergence is very slow. The improved
dual decomposition approach however converges to the optimal dual
and primal variables in only $150$ iterations, within an accuracy of
$0.05\%$. The optimal transmit powers are shown in Figure
\ref{fig:vdsl6_powers} for illustration.\\

\begin{figure}[h]
\centering
\includegraphics[width=0.5\columnwidth]{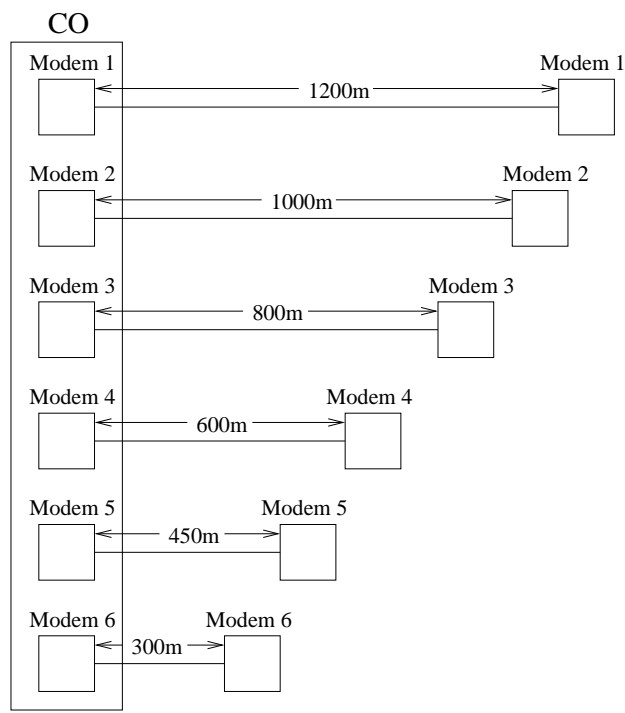}
\caption{6-user VDSL upstream scenario}
\label{fig:vdsl6}
\end{figure}

\begin{figure}[h]
\centering
\includegraphics[width=0.6\columnwidth]{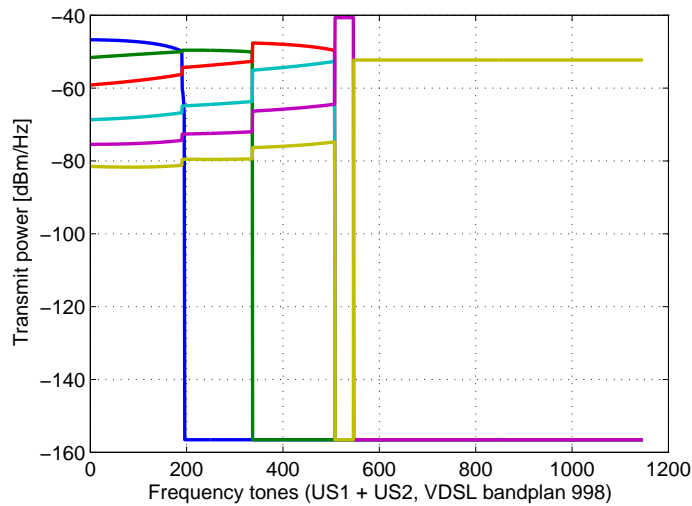}
\caption{Optimal transmit powers for DSL scenario of Fig.
\ref{fig:vdsl6} obtained using the improved dual decomposition
approach. Blue, green, red, cyan, magenta and yellow curves
correspond to transmit powers of users with line length $1200$m,
$1000$m, $800$m, $600$m, $450$m and $300$m respectively.}
\label{fig:vdsl6_powers}
\end{figure}

The VDSL upstream scenario of Figure \ref{fig:vdsl6sym} consists of
a six-line cable bundle with a subset of three strong symmetric
crosstalkers, namely the set of lines with length $300$m. The
standard subgradient approaches \cite{tsiaflakis_bbosb}
\cite{dual_journal} fail to converge to the optimal dual variables.
The presence of the strong symmetric crosstalkers significantly
slows down the convergence, as it can lead to multiple
globally optimal solutions for particular values of the dual
variables. Here, stepsize selection is very crucial as a small
change in dual variables can lead to a large change in primal
variables, as also explained in Section \ref{sec:dualprimal}. The
improved dual decomposition approach converges to the optimal
dual variables in only $150$ iterations, but does not succeed in
obtaining the primal optimal variables, because of the existence of
multiple globally optimal solutions (i.e. optimal transmit powers)
for optimal dual variables that do not satisfy the user total power
constraints. More specifically for this scenario, for the obtained
optimal dual variables, the obtained transmit powers jump to
different solutions, with total powers
$\{P^1,P^2,P^3\}=\{P^{1,\mathrm{tot}},P^{2,\mathrm{tot}},P^{3,
\mathrm {tot}}\}$, and $\{P^4,P^5,P^6\} \in
\big\{
\{3P^{\mathrm{tot}},A,A\},\{A,3P^{\mathrm{tot}},A\},\{A,A,3P^{
\mathrm{tot}}\}\big\}$, with $A$ being very small. These
primal solutions are shown in Figures
\ref{fig:vdsl6sym_NOinterleaving_three},
\ref{fig:vdsl6sym_NOinterleaving_two}
and \ref{fig:vdsl6sym_NOinterleaving_one} . One can observe that
in the low and medium frequency range (used tones 1-727), the users
with line lengths $1200$ m, $900$ m and $600$ m are active. In this
frequency range the strong crosstalkers will back-off and transmit
at small similar transmit powers corresponding to a total power
equal to $A$. However in the high frequency range 
(used tones 727-1147) where the users with line lengths $1200$ m,
$900$ m and $600$ are switched off, the three strong crosstalkers
will compete, where only one user can be active in each tone $k$
because of the significant crosstalk interference \cite{luo_dsm}.
As explained in Section \ref{sec:dualprimal}, typical DSM algorithm
implementations will select the same active user for each of these
tones, namely the user that corresponds to the smallest
dual variable, where the dual variable can be seen as a penalty. So
instead of dividing the total power over the three users equally,
which would lead to a primal solution satisfying the per-user total
power constraints, one user gets all power, leading to $P^{n}=3
P^{n,\mathrm{tot}}$ for user $n$ and $P^{m}=A$ for users $m \neq n$.
Note that this prevents convergence to the optimal primal variables
satisfying the per-user total power constraints.

However, when applying the proposed interleaving procedure
(\ref{eq:interleaving}), as proposed in Section
\ref{sec:dualprimal}, together with the improved dual decomposition
approach, we can observe a very fast convergence both in primal and
dual variables. Convergence is achieved in only $150$ iterations,
within an accuracy of $0.05\%$. The obtained optimal transmit powers
are shown in Figure \ref{fig:vdsl6sym_interleaving}. In the
frequency range between tone $728$ and tone $1147$, one can observe
the interleaving effect. In Figure
\ref{fig:vdsl6sym_interleaving_tones} this is zoomed in for tones
$970$ up to $975$.

\textbf{Remark: }
In the practical implementation the first step of the interleaving
procedure is changed to `all best solutions that are 99.9\% close to
each other'. This is to prevent that the procedure is only active
when the dual variables are exactly the same. The overall effect of
this is a negligible noise on the transmit powers as can be seen in
Figure \ref{fig:vdsl6sym_interleaving}.

\textbf{Remark: }
Note that applying the interleaving procedure combined with the
improved dual decomposition approach for the scenarios in Figures
\ref{fig:vdsl4} and \ref{fig:vdsl6}, also leads to a faster
convergence in both dual and primal variables. 

%

\begin{figure}[h]
\centering
\includegraphics[width=0.5\columnwidth]{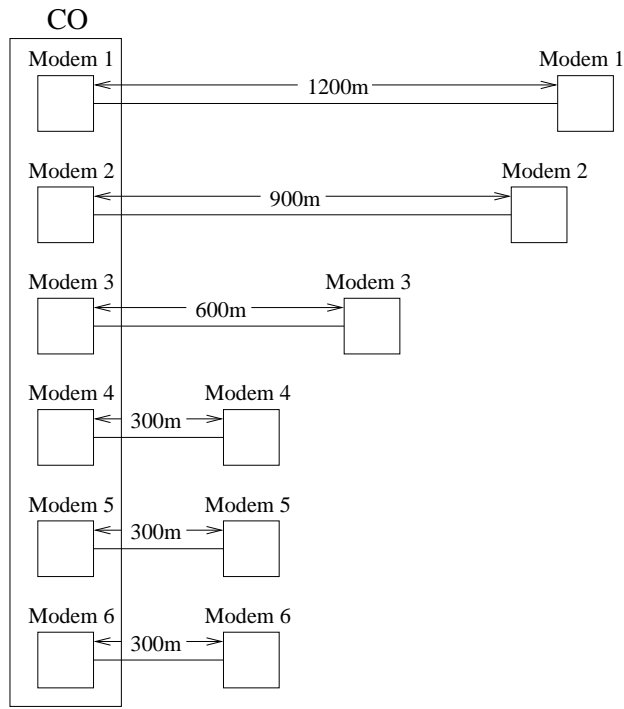}
\caption{6-user VDSL upstream scenario with subset of strong
symmetric crosstalkers}
\label{fig:vdsl6sym}
\end{figure}

\begin{figure}[h]
\centering
\includegraphics[width=0.6\columnwidth]
{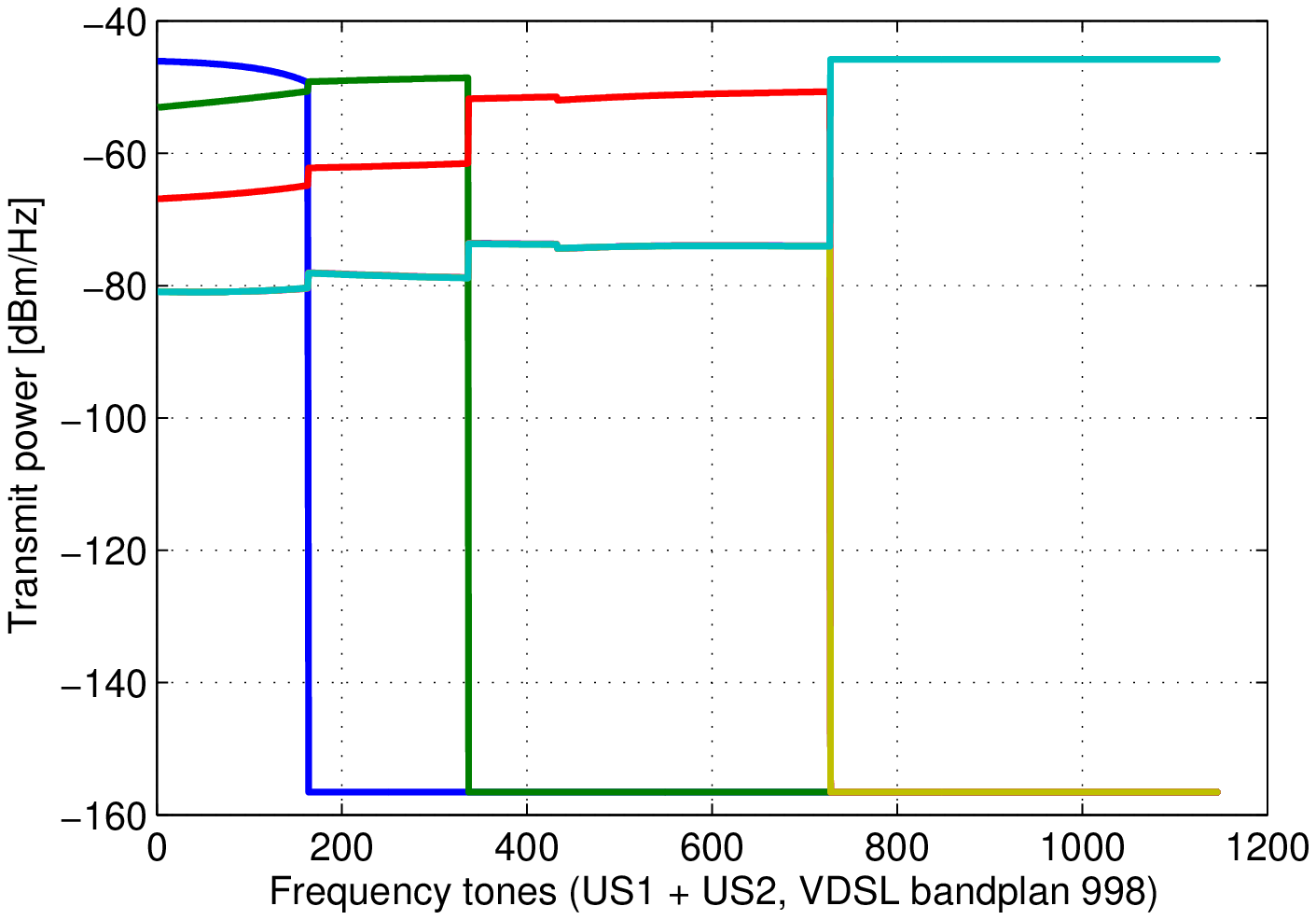}
\caption{Transmit powers for DSL scenario of Fig.
\ref{fig:vdsl6sym} for optimal dual variables $\mathbf{\lambda}^*$
and with user total powers $\{P^1,P^2,P^3,P^4,P^5,P^6\} =
\{P^{1,\mathrm{tot}},P^{2,\mathrm{tot}},P^{3,\mathrm{tot}},
3 P^{4,\mathrm{tot}},A,A\}$, where $A << P^{4,\mathrm{tot}}$.
Blue, green, red, cyan, magenta, yellow
curves correspond to transmit powers of users 1,2,3,4,5 and 6
respectively.}
\label{fig:vdsl6sym_NOinterleaving_three}
\end{figure}

\begin{figure}[h]
\centering
\includegraphics[width=0.6\columnwidth]{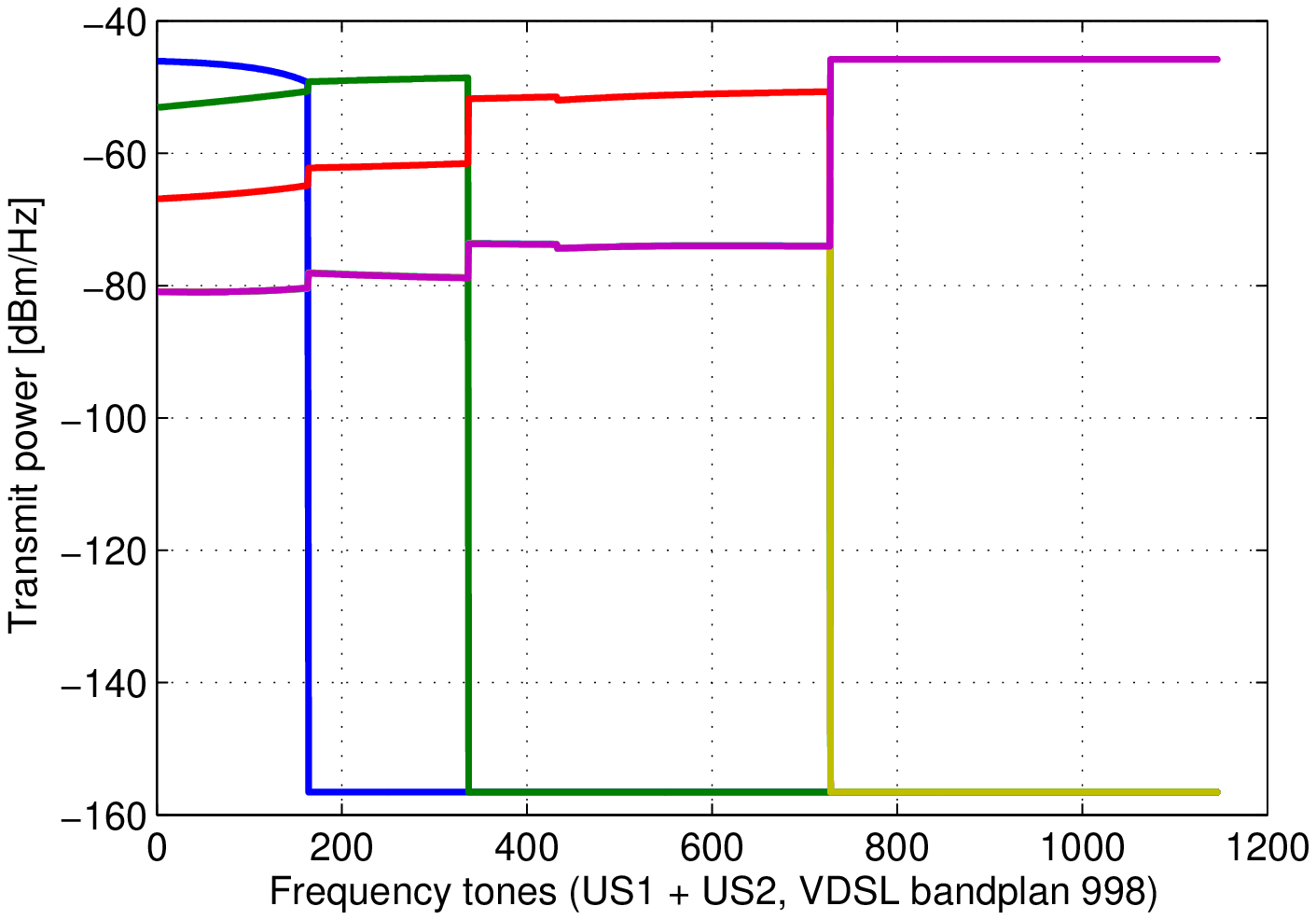}
\caption{Transmit powers for DSL scenario of Fig.
\ref{fig:vdsl6sym} for optimal dual variables $\mathbf{\lambda}^*$
and with user total powers $\{P^1,P^2,P^3,P^4,P^5,P^6\} =
\{P^{1,\mathrm{tot}},P^{2,\mathrm{tot}},P^{3,\mathrm{tot}},A,
3 P^{5,\mathrm{tot}},A\}$, where $A << P^{5,\mathrm{tot}}$. Blue,
green, red, cyan, magenta, yellow
curves correspond to transmit powers of users 1,2,3,4,5 and 6
respectively.}
\label{fig:vdsl6sym_NOinterleaving_two}
\end{figure}

\begin{figure}[h]
\centering
\includegraphics[width=0.6\columnwidth]{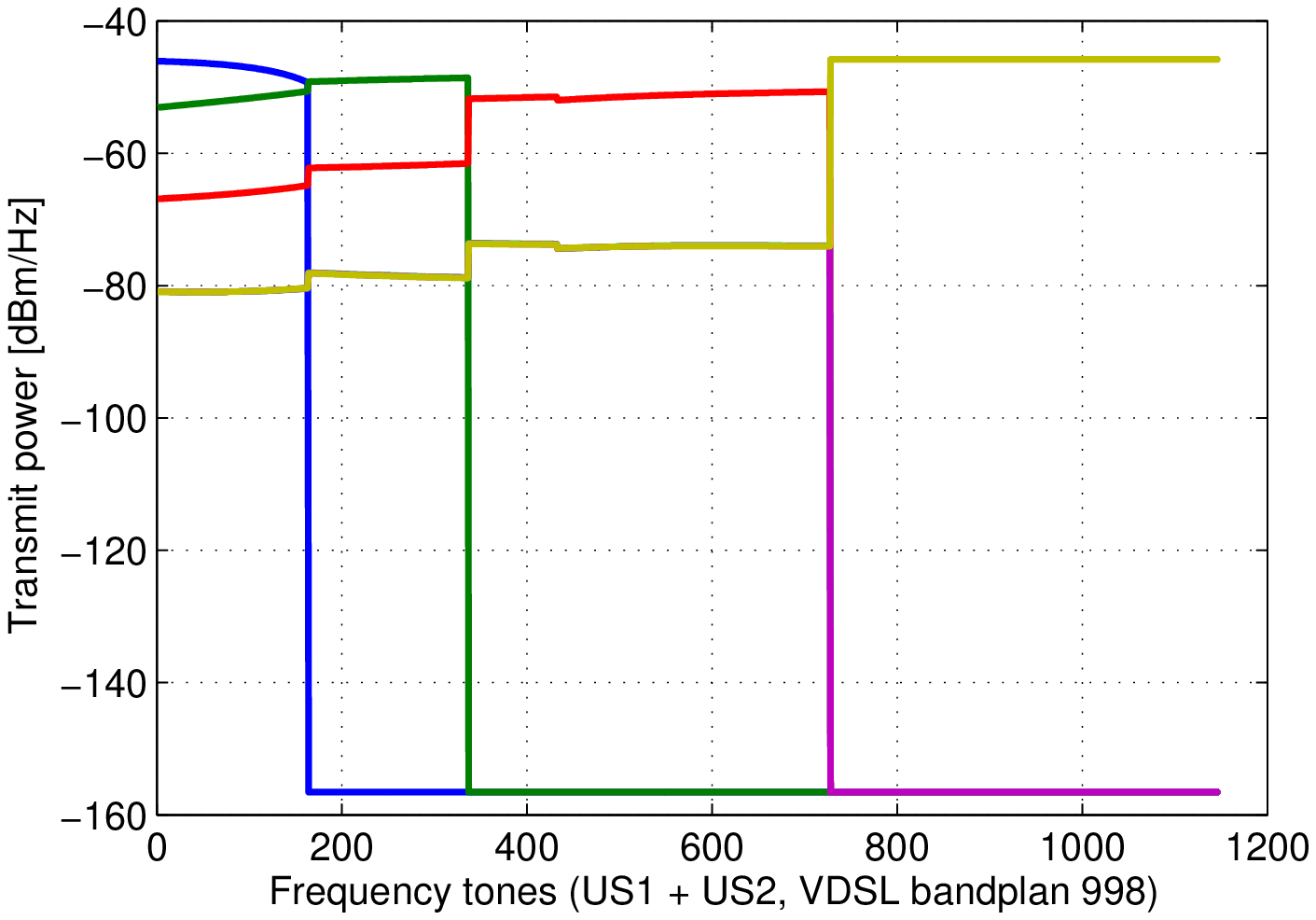}
\caption{Transmit powers for DSL scenario of Fig.
\ref{fig:vdsl6sym} for optimal dual variables $\mathbf{\lambda}^*$
and with user total powers $\{P^1,P^2,P^3,P^4,P^5,P^6\} =
\{P^{1,\mathrm{tot}},P^{2,\mathrm{tot}},P^{3,\mathrm{tot}},A,A,
3 P^{6,\mathrm{tot}}\}$, where $A << P^{6,\mathrm{tot}}$. Blue,
green, red, cyan, magenta, yellow curves correspond to transmit
powers of users 1,2,3,4,5 and 6 respectively.}
\label{fig:vdsl6sym_NOinterleaving_one}
\end{figure}

\begin{figure}[h]
\centering
\includegraphics[width=0.6\columnwidth]
{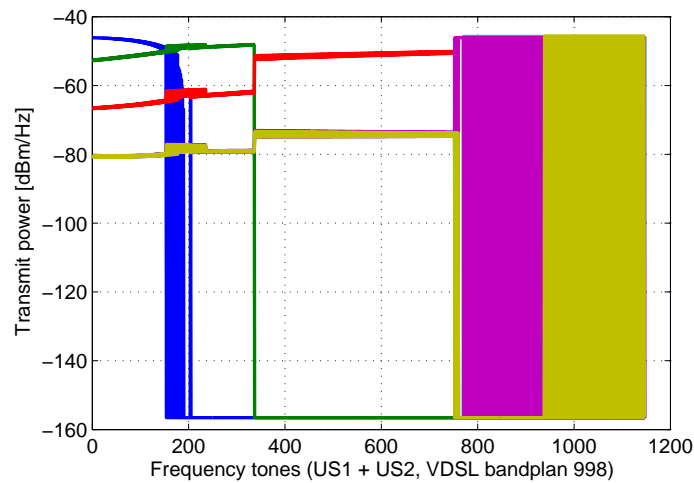}
\caption{Transmit powers for scenario of DSL scenario of Fig.
\ref{fig:vdsl6sym} obtained using improved dual decomposition
approach with the interleaving procedure (\ref{eq:interleaving}).
Blue, green, red, cyan, magenta, yellow curves correspond to
transmit powers of users 1,2,3,4,5 and 6 respectively.}
\label{fig:vdsl6sym_interleaving}
\end{figure}

\begin{figure}[h]
\centering
\includegraphics[width=0.6\columnwidth]
{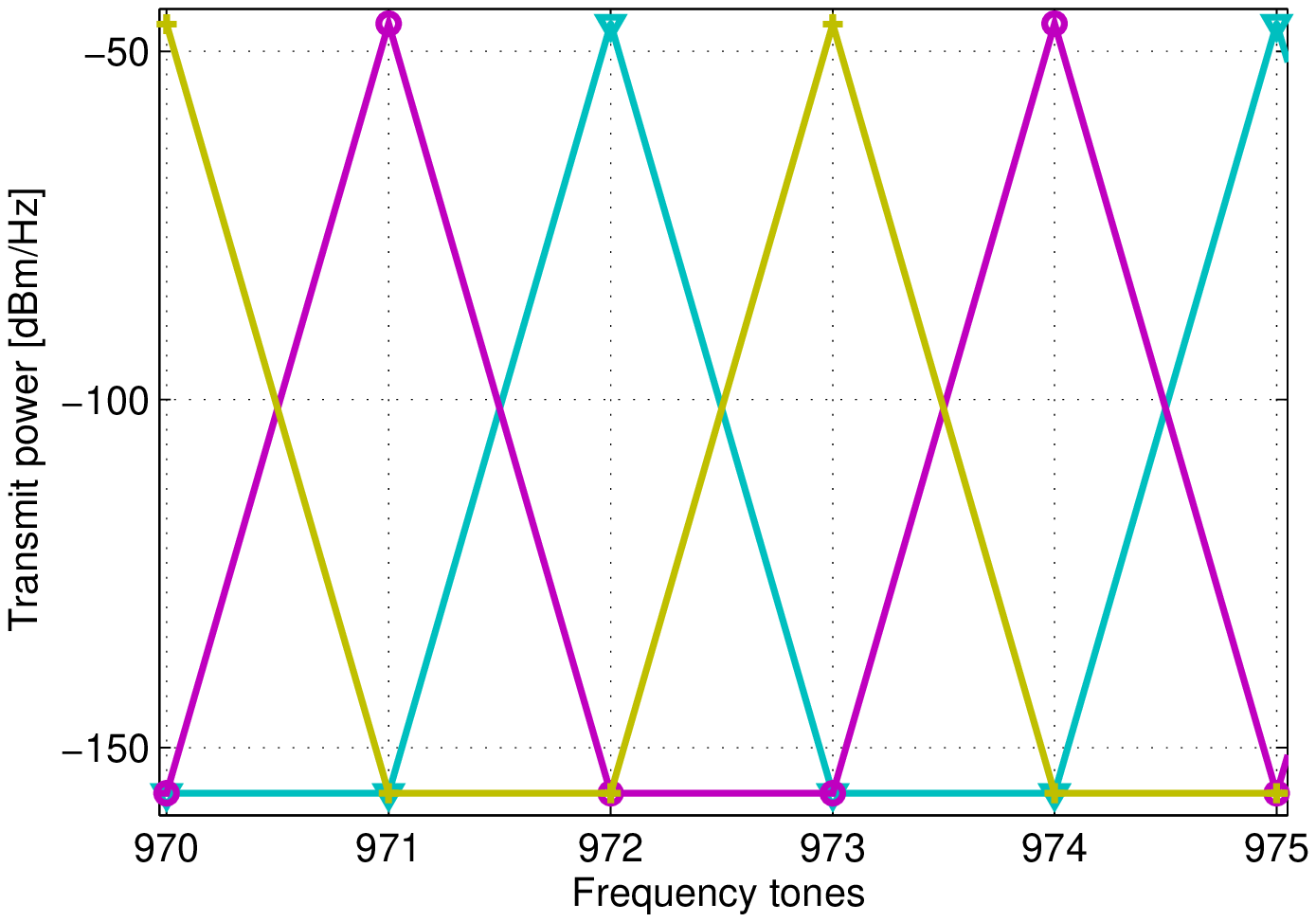}
\caption{Zoom in on tones $970$ to $975$ of Fig.
\ref{fig:vdsl6sym_interleaving}.}
\label{fig:vdsl6sym_interleaving_tones}
\end{figure}

\section{Conclusion}\label{sec:conclusion}
Dynamic spectrum management has been recognized as a key technology
to significantly improve the performance of DSL broadband access
networks by mitigating the impact of crosstalk
interference. Existing DSM algorithms use a standard subgradient
based dual decomposition approach to tackle the corresponding
nonconvex optimization problems. However, this standard approach is
often found to lead to extremely slow convergence or even no
convergence at all. Especially for multiuser VDSL scenarios with
subsets of strong symmetric crosstalkers significant convergence
problems are observed because (1) the stepsize selection procedure
of the subgradient updates is very critical, and (2) because special
care must be taken when recovering the optimal transmit powers from
the optimal dual solution. This paper proposes an improved dual
decomposition approach, which consists of an optimal gradient based
scheme with an automatic optimal stepsize selection removing the
need for a tuning strategy. With this approach it is shown how the
convergence of current state-of-the-art DSM algorithms, based on
iterative convex approximations, is improved by one order of
magnitude. The improved dual decomposition approach is also applied
to other DSM algorithms (OSB, ISB, ASB, (MS)-DSB, MIW). The addition
of an extra interleaving procedure for recovering the optimal
transmit powers from the dual optimal solution furthermore improves
the convergence of the proposed approach. Simulation results
demonstrate that significant convergence speed ups are obtained
using the proposed improved dual decomposition approach.

\bibliographystyle{IEEEtran}
\bibliography{bibliographyshort}

\begin{thebibliography}{10}
\providecommand{\url}[1]{#1}
\csname url@samestyle\endcsname
\providecommand{\newblock}{\relax}
\providecommand{\bibinfo}[2]{#2}
\providecommand{\BIBentrySTDinterwordspacing}{\spaceskip=0pt\relax}
\providecommand{\BIBentryALTinterwordstretchfactor}{4}
\providecommand{\BIBentryALTinterwordspacing}{\spaceskip=\fontdimen2\font plus
\BIBentryALTinterwordstretchfactor\fontdimen3\font minus
  \fontdimen4\font\relax}
\providecommand{\BIBforeignlanguage}[2]{{%
\expandafter\ifx\csname l@#1\endcsname\relax
\typeout{** WARNING: IEEEtran.bst: No hyphenation pattern has been}%
\typeout{** loaded for the language `#1'. Using the pattern for}%
\typeout{** the default language instead.}%
\else
\language=\csname l@#1\endcsname
\fi
#2}}
\providecommand{\BIBdecl}{\relax}
\BIBdecl

\bibitem{dsldominates}
{DSL Forum (www.dslforum.org)}, ``{DSL dominates global broadband subscriber
  growth},'' Tech. Rep., {March} {2007}, internet: \url{
  http://www.broadband-forum.org/news/download/pressreleeases/YE0
  6_Release.pdf}.

\bibitem{DSM_Song}
{K.B. Song, S.T. Chung, G. Ginis, J.M. Cioffi}, ``{Dynamic spectrum management
  for next-generation DSL systems},'' \emph{{IEEE Communications Magazine}},
  vol.~{40}, no.~{10}, pp. {101--109}, {Oct.} {2002}.

\bibitem{optSpectrManagementJournal}
{R.~Cendrillon, W.~Yu, M.~Moonen, J.~Verlinden and T.~Bostoen}, ``{Optimal
  multiuser spectrum balancing for digital subscriber lines},'' \emph{{IEEE
  Trans. Comm.}}, vol.~54, no.~5, pp. 922--933, {May} {2006}.

\bibitem{dual_journal}
{W.~Yu and R.~Lui}, ``{Dual methods for nonconvex spectrum optimization of
  multicarrier systems},'' \emph{{IEEE Trans. on Comm.}}, vol.~{54}, no.~{7},
  {July} {2006}.

\bibitem{tsiaflakis_bbosb}
{P.~Tsiaflakis, J.~Vangorp, M.~Moonen, J.~Verlinden}, ``{A low complexity
  optimal spectrum balancing algorithm for digital subscriber lines},''
  \emph{{Signal Processing}}, vol.~{87}, no.~{7}, pp. {1735--1753}, July
  {2007}.

\bibitem{Tsiaflakis2009a}
{P.~Tsiaflakis, Y.~Yi, M.~Chiang, M.~Moonen}, ``{Green DSL: Energy-Efficient
  DSM},'' in \emph{{accepted for publication in IEEE International Conference
  on Communications (ICC 2009)}}, {June} {2009}.

\bibitem{energyDSLnordstrom}
{M.~Wolkerstorfer, D.~Statovci, and T.~Nordström}, ``Dynamic spectrum
  management for energy-efficient transmission in dsl,'' in \emph{{The Eleventh
  IEEE International Conference on Communications Systems (ICCS 2008),
  Guangzhou, China}}, November 2008.

\bibitem{greenCopper}
{J.~M.~Cioffi, S.~Jagannathan, W.~Lee, H.~Zou, A.~Chowdhery, W.~Rhee, G.~Ginis,
  P.~Silverman}, ``Greener copper with dynamic spectrum management,'' in
  \emph{AccessNets, Las Vegas, NV, USA}, Oct 2008.

\bibitem{DSMluo}
{Z.~Q.~Luo, S.~Zhang}, ``{Dynamic Spectrum Management: Complexity and
  Duality},'' \emph{{IEEE Journal of Selected Topics in Signal Processing}},
  vol.~{2}, no.~{1}, pp. {57--73}, {Feb.} {2008}.

\bibitem{Tsiaflakis2008}
{P.~Tsiaflakis, Y.~Yi, M.~Chiang, M.~Moonen}, ``{Throughput and delay of DSL
  dynamic spectrum management with dynamic arrivals},'' in \emph{{IEEE Global
  Telecommunications Conference (GLOBECOM)}}, {November} {2008}, pp. {1--5}.

\bibitem{dsb}
{P.~Tsiaflakis, M.~Diehl, M.~Moonen}, ``{Distributed Spectrum Management
  Algorithms for Multiuser DSL Networks},'' \emph{{IEEE Transactions on Signal
  Processing}}, vol.~{56}, no.~{10}, pp. {4825--4843}, Oct {2008}.

\bibitem{PBnB}
{Y.~Xu, T.~Le-Ngoc, S.~Panigrahi}, ``Global concave minimization for optimal
  spectrum balancing in multi-user dsl networks,'' \emph{IEEE Transactions on
  Signal Processing}, vol.~56, no.~7, pp. pp. 2875--2885, July 2008.

\bibitem{ISB_Raphael}
{R. Cendrillon, M. Moonen}, ``{Iterative spectrum balancing for digital
  subscriber lines},'' in \emph{{IEEE Int. Conf. on Communications}}, vol.~3,
  no.~3, May 2005, pp. 1937--1941.

\bibitem{scale}
{J.~Papandriopoulos and J.~S.~Evans}, ``{Low-complexity distributed algorithms
  for spectrum balancing in multi-user DSL networks},'' in \emph{{IEEE Int.
  Conf. on Communications}}, vol.~7, June 2006, pp. 3270--3275.

\bibitem{MIW}
{W.~Yu}, ``{Multiuser water-filling in the presence of crosstalk},'' in
  \emph{{Information Theory and Applications (ITA)}}, {Feb.} {2007}.

\bibitem{ASB}
{R.~Cendrillon, J.~Huang, M.~Chiang, M.~Moonen}, ``{Autonomous spectrum
  balancing for digital subscriber lines},'' \emph{{IEEE Trans. on Signal
  Processing}}, vol.~{55}, no.~{8}, pp. {4241--4257}, {August} {2007}.

\bibitem{cioffi_DSB}
{S.~Jagannathan and J.~M.~Cioffi}, ``{Distributed Adaptive Bit-loading for
  Spectrum Optimization in Multi-user Multicarrier Systems},'' \emph{Elsevier
  Physical Communication}, vol.~1, no.~1, pp. pp. 40--59, 2008.

\bibitem{iterativeWaterfilling}
{W.~Yu, G.~Ginis, J.~Cioffi}, ``{Distributed multiuser power control for
  digital subscriber lines},'' \emph{{IEEE J. Sel. Area. Comm.}}, vol.~20,
  no.~5, pp. 1105--1115, Jun. 2002.

\bibitem{Ion2008}
{I.~Necoara, J.A.K.~Suykens}, ``{Application of a smoothing technique to
  decomposition in convex optimization},'' \emph{{IEEE Transactions on
  Automatic Control}}, vol.~{11}, no.~{53}, pp. {2674--2679}, {2008}.

\bibitem{Nesterov2004}
{Y.~Nesterov}, \emph{{Introductory Lectures on Convex Optimization: A Basic
  Course}}, {Kluwer}, Ed.\hskip 1em plus 0.5em minus 0.4em\relax {Boston},
  {2004}.

\bibitem{UnderstandingDSL}
{T. Starr, J. M. Cioffi, P. J. Silverman}, \emph{{Understanding digital
  subscriber lines}}.\hskip 1em plus 0.5em minus 0.4em\relax {Prentice Hall},
  1999.

\bibitem{ISB_WeiYu}
{R. Lui and W. Yu}, ``{Low-complexity near-optimal spectrum balancing for
  digital subscriber lines},'' in \emph{{IEEE Int. Conf. on Communications}},
  vol.~3, no.~3, May 2005, pp. {1947--1951}.

\bibitem{siwf}
{W.~Lee, Y.~Kim, M.~H.~Brady and J.~M.~Cioffi}, ``{Band-preference dynamic
  spectrum management in a DSL environment},'' in \emph{{IEEE Glob. Telecomm.
  Conf.}}, Nov. 2006, pp. 1--5.

\bibitem{Chiang2007}
{M.~Chiang, C.W.~Tan, D.P.~Palomar, D.~O'Neill, D.~Julian}, ``{Power control by
  geometric programming},'' \emph{{IEEE Transactions on Wireless
  Communications}}, vol.~{6}, no.~{7}, pp. {2640--2651}, {July} {2007}.

\bibitem{KonLeo:96}
S.~Kontogiorgis, R.~D. Leone, and R.~Meyer, ``Alternating direction splitings
  for block angular parallel optimization,'' \emph{Journal of Optimization
  Theory and Applications}, vol.~90, no.~1, pp. 1--29, 1996.

\bibitem{CheTeb:94}
G.~Chen and M.~Teboulle, ``A proximal-based decomposition method for convex
  minimization problems,'' \emph{Mathematical Programming (A)}, vol.~64, pp.
  81--101, 1994.

\bibitem{Spi:85}
J.~E. Spingarn, ``Applications of the method of partial inverses to convex
  programming: decomposition,'' \emph{Mathematical Programming (A)}, vol.~32,
  pp. 199--223, 1985.

\bibitem{Tsiaflakis2008a}
{P.~Tsiaflakis, C.W.~Tan, Y.~Yi, M.~Chiang, M.~Moonen}, ``{Optimality
  certificate of dynamic spectrum management in multi-carrier interference
  channels},'' in \emph{{IEEE International Symposium on Information Theory}},
  {July} {2008}, pp. {1298--1302}.

\bibitem{Boyd2004}
{L.~Xiao, M.~Johansson, S.~Boyd}, ``{Simultaneous Routing and Resource
  Allocation via Dual Decomposition},'' \emph{{IEEE Transactions on
  Communications}}, vol.~{52}, no.~{7}, pp. 1136--1144, July {2004}.

\bibitem{luo_dsm}
{S.~Hayashi, Z.~Q.~Luo}, ``{Dynamic spectrum management: When is FDMA sum-rate
  optimal?}'' in \emph{{IEEE Int. Conf. on Acoustics, Speech and Sign. Proc.}},
  vol.~3, April 2007, pp. III--609 -- III--612.

\end{thebibliography}

\end{document}